\documentclass[reqno]{amsart}

\usepackage[english]{babel}
\usepackage[utf8]{inputenc} 

\usepackage{amsfonts,amssymb,amsmath,amsthm}
\numberwithin{equation}{section}
\usepackage{dsfont}
\usepackage{multirow}
\usepackage{url}             
\usepackage{mathrsfs}
\usepackage{mathtools}
\mathtoolsset{showonlyrefs}  
\usepackage{nicefrac}
\usepackage{paralist}
\usepackage[shortlabels,inline]{enumitem}

\setlist[enumerate]{leftmargin=1cm,topsep=1mm}
\setlist[itemize]{leftmargin=1cm,topsep=1mm}

\usepackage[pdftex,
	pdftitle={Dirichlet problems associated to abstract nonlocal 
		space--time differential operators},
	pdfauthor={J.~Willems},
	bookmarksopen,
	colorlinks,
	linkcolor=black,
	urlcolor=black,
	citecolor=black
	]{hyperref}

\newcommand{\bbR}{\mathbb{R}}

\newcommand{\R}{\mathbb{R}}
\newcommand{\C}{\mathbb{C}}
\newcommand{\N}{\mathbb{N}}

\newcommand{\cA}{\mathcal{A}}
\newcommand{\cB}{\mathcal{B}}

\newcommand{\cD}{\mathcal{D}}

\newcommand{\cL}{\mathcal{L}}

\newcommand{\cW}{\mathcal{W}}

\newcommand{\fD}{\mathfrak{D}}
\newcommand{\fI}{\mathfrak{I}}

\newcommand{\norm}[2]{     \| #1       \|_{ #2 }}

\newcommand{\scalar}[2]{     \langle #1       \rangle_{ #2 }}

\renewcommand{\Re}{\operatorname{Re}}
\renewcommand{\Im}{\operatorname{Im}}

\newcommand{\rd}{\mathop{}\!\mathrm{d}}
\newcommand{\from}{\colon}
\newcommand{\clos}[1]{\overline{ #1 }} 
\newcommand{\id}{\operatorname{Id}}
\newcommand{\dom}[1]{\mathsf D(#1)}

\newcommand{\LO}{\mathscr L}
\newcommand{\Lloc}{L_{\mathrm{loc}}}
\newcommand{\Cub}{C_{\mathrm{ub}}}
\newcommand{\Cb}{C_{\mathrm{b}}}

\newtheorem{lemma}{Lemma}[section]
\newtheorem{proposition}[lemma]{Proposition}
\newtheorem{theorem}[lemma]{Theorem}
\newtheorem{corollary}[lemma]{Corollary}

\theoremstyle{remark}
\newtheorem{remark}[lemma]{Remark}

\theoremstyle{definition}
\newtheorem{definition}[lemma]{Definition}
\newtheorem{assumption}[lemma]{Assumption}
\newtheorem{example}[lemma]{Example}

\begin{document}

\author{Joshua Willems}

\address[Joshua Willems]{Delft Institute of Applied Mathematics\\
	Delft University of Technology\\
	P.O.~Box 5031 \\ 
	2600 GA Delft \\
	The Netherlands.}

\email{j.willems@tudelft.nl}


\title[Abstract nonlocal space--time Dirichlet problems]{%
Dirichlet problems associated to abstract nonlocal 
space--time differential operators}




\date{}

\begin{abstract}
	Let the abstract fractional space--time operator $(\partial_t + A)^s$ be given,	where $s \in (0,\infty)$ and $-A \colon \mathsf{D}(A) \subseteq X \to X$ is a linear operator generating a uniformly bounded strongly measurable semigroup $(S(t))_{t\ge0}$ on a complex Banach space $X$.
	We consider the corresponding  Dirichlet problem of finding	$u \colon \mathbb{R} \to X$ such that 
	\[
	\left\lbrace 
	\begin{aligned}
		(\partial_t + A)^s u(t) &= 0, &  t &\in (t_0, \infty), \\
		u(t) &= g(t), &  t &\in (-\infty, t_0],
	\end{aligned}
	\right.
	\]
	for given $t_0 \in \mathbb{R}$ and $g \colon (-\infty,t_0] \to X$. 
	We define the concept of $L^p$-solutions, to which we associate a mild solution formula
	which expresses $u$ in terms of $g$ and $(S(t))_{t\ge0}$ and generalizes the well-known variation of constants formula for the mild solution to the	abstract Cauchy problem $u' + Au = 0$ on $(t_0, \infty)$ with $u(t_0) = x \in \overline{\mathsf{D}(A)}$. 
	Moreover, we include a comparison to analogous solution concepts arising from Riemann--Liouville and Caputo type initial value problems.
\end{abstract}

\maketitle


\section{Introduction}
\label{section:intro}

\subsection{Background and motivation}
\label{section:intro:background-motivation}

Space--time nonlocal problems involving fractional powers of a parabolic
operator arise in physics, biology, probability theory and statistics. 
The flat parabolic Signorini problem and certain models for semipermeable
membranes can be formulated as obstacle problems for the fractional heat
operator $(\partial_t - \Delta)^s$, where $s \in (0,1)$ and $\Delta$ denotes 
the Laplacian, acting on
functions $u \from J \times \cD \to \R$ for a given time interval 
$J \subseteq \bbR$
and a connected non-empty open spatial domain $\cD \subseteq \R^d$,
see e.g.~\cite{AthanasopoulosCaffarelliMilakis2018, Stinga2017}. 
In the context of
continuous time random walks, equations of the form 
$(\partial_t - \Delta)^s u = f$
for $f \from J \times \cD \to \R$
are considered examples of \emph{master equations} governing the (non-separable)
joint probability
distribution of jump lengths and waiting times~\cite{CaffarelliSilvestre2014}.
The case where $f$ is replaced by
spatiotemporal Gaussian noise $\dot \cW$ has applications to the statistical
modeling of spatial and temporal dependence in data: The resulting class of 
fractional 
parabolic \emph{stochastic} partial differential equations (SPDEs) has 
been proposed and
analyzed in~\cite{KW2023, LBBKR2024} as a spatiotemporal
generalization of the \emph{SPDE approach} to spatial statistical
modeling,
which was initiated by Lindgren, Rue and
Lindstr\"om~\cite{LindgrenRueLindstroem2011}
and has subsequently gained widespread popularity~\cite{LindgrenBolinRue2022}. 

After \cite{Stinga2017} and~\cite{NystroemSande2016} independently generalized
the Caffarelli--Silvestre extension approach from the fractional elliptic to the
parabolic setting, there has been a surge of literature on
space--time nonlocal problems involving fractional powers of $\partial_t + L$
for more general elliptic operators $L$ acting on functions $u \from \cD \to
\R$, see for instance~\cite{BanerjeeGarofalo2018, BanerjeeGarofalo2023,
	BiswasDLCStinga2021, BiswasStinga2021, FjellstroemNystromWang2023,
	LaiLinRueland2020, LitsgardNystroem2023}. In 
particular, in~\cite[Remark~1.2]{Stinga2017}, the
\emph{natural Dirichlet problem
for the nonlocal space--time operator $(\partial_t + L)^s$}, given 
by
\begin{equation}\label{eq:natural-Dirichlet-problem-xt}
	\left\lbrace
	\begin{aligned}
		(\partial_t + L)^s u(t, x)
		&=
		f(t,x), &
		&(t,x) \in J \times \cD,
		\\
		u(t,x) &= g(t, x), 
		&&(t,x) \in \R^{d+1} \setminus (J \times \cD),
	\end{aligned}
	\right.
\end{equation}
is introduced,
where $g \from \R^{d+1} \setminus (J \times \cD) \to \R$ is a given function
prescribing the values of $u$ outside of the spatiotemporal region 
$J \times\cD$. 
The definition of $(\partial_t + L)^s$, given in 
Section~\ref{section:prelims:frac-par-calc} below, generalizes that of the 
Riemann--Liouville fractional time derivative $\partial_t^s$ (i.e., the case 
where $L = 0$) using the theory of semigroups.
Equations 
involving only a fractional time derivative have been studied widely; see for 
instance the monographs~\cite{KST2006, SKM1993, Podlubny1999, GalWarma2020} for 
an 
introduction to the subject.

In the integer-order case, the space--time differential operator is 
local, so 
that the analog to~\eqref{eq:natural-Dirichlet-problem-xt} is
an initial boundary value problem. Identifying any $u \from J \times \cD 
\to 
\bbR$ with $u \from J \to
X$, where $J \coloneqq (t_0, \infty)$ and $X$ is a Banach space to be 
thought of as containing functions from 
$\cD$ 
to 
$\R$, the corresponding infinite-dimensional initial value problem for $s 
= 1$ is the \emph{abstract Cauchy problem}
\begin{equation}\label{eq:ACP}
	\left\lbrace
	\begin{aligned}
		(\partial_t + A) u(t)
		&=
		f(t), 
		\quad
		t \in J,
		\\
		u(t_0) &= x \in X.
	\end{aligned}
	\right.
\end{equation}
Here, $A \from \dom{A} \subseteq X \to X$ is a linear operator, whose domain
$\dom{A}$ can be used to encode (Dirichlet) boundary conditions, and $f \from J 
\to X$ is a given forcing function. 
Although there exist various definitions of solutions to~\eqref{eq:ACP} 
(e.g., mild, strong and $L^p$-solutions, see Section~\ref{sec:11st-order-ACP}), 
the main focus of this article is 
on mild solutions. The mild solution to~\eqref{eq:ACP} can only be defined
under the assumption that $-A$ is the infinitesimal generator of a suitably 
regular semigroup $(S(t))_{t\ge0}$ of bounded linear operators on $X$, 
see Section~\ref{sec:prelims:semigroups}.
Intuitively, the relation between $(S(t))_{t\ge0}$ and $A$ can be viewed as a 
generalization of the way in which the matrix exponential $S(t) = e^{-tA}$ is 
associated to a matrix $A$.
If, moreover, the right-hand side $f$ is sufficiently (Bochner) integrable, 
then the \emph{mild solution} of~\eqref{eq:ACP} is defined by
\begin{equation}\label{eq:variation-of-constants}
	u(t) \coloneqq S(t - t_0)x + \int_{t_0}^t S(t-\tau) f(\tau) \rd \tau, 
	\quad 
	t \in J,
\end{equation}
which is commonly known as the \emph{variation of constants formula}, 
again by analogy with the finite-dimensional (matrix) case.

In this work we consider an abstract counterpart 
of~\eqref{eq:natural-Dirichlet-problem-xt} in the setting 
of~\eqref{eq:ACP}, namely the following Dirichlet problem
for $(\partial_t + A)^s$ with
$s \in (0,\infty) \setminus \N$:
\begin{equation}\label{eq:natural-Dirichlet-abstract}
	\left\lbrace
	\begin{aligned}
		(\partial_t + A)^s u(t)
		&=
		0, &
		&t \in (t_0, \infty),
		\\
		u(t) &= g(t), 
		&& t \in (-\infty, t_0],
	\end{aligned}
	\right.
\end{equation}
where $g \from (-\infty, t_0] \to X$. We restrict ourselves to $J = (t_0, 
\infty)$
since $(\partial_t + A)^s u(t)$ depends only on the values of $u$ 
to the left of
$t \in \R$, see Section~\ref{section:prelims:frac-par-calc} below.
Moreover, we only consider $f \equiv 0$ since the problem is linear in 
$u$
and the mild solution formula for $f \not\equiv 0$ and $g \equiv 0$ (or 
$J = \R$) is known to be given by a Riemann--Liouville type fractional 
parabolic integral, 
cf.~\cite[Theorem~1.17]{Stinga2017}. 
We will define the concept of an \emph{$L^p$-solution} to~\eqref{eq:natural-Dirichlet-abstract},
and show that it can be expressed in terms of $g$ and 
$(S(t))_{t\ge0}$ in the following way:
\begin{equation}\label{eq:mildsol-Dirichlet}
	\begin{aligned}
		u(t)
		\coloneqq
		{}&
		\frac{\sin(\pi \{s\})}{\pi}
		\int_0^\infty 
		\frac{\tau^{-\{s\}}}{\tau + 1} S((t-t_0)(\tau + 1)) g(t_0 - (t-t_0)\tau)
		\rd \tau
		\\&+
		\sum_{k=1}^{\lfloor s \rfloor}
		\frac{(t-t_0)^{\{s\} + k - 1}}{\Gamma(\{s\} + k)}
		S(t-t_0) 
		[(\partial_t + A)^{\{s\} + k - 1} g](t_0), 
		\quad 
		t \in (t_0, \infty),
	\end{aligned}
\end{equation}
where $\Gamma$ denotes the gamma function and
$s = \lfloor s \rfloor + \{s\}$ for $\lfloor s \rfloor \in \N_0$ and 
${\{s\} \in (0,1)}$.
This formula
generalizes~\eqref{eq:variation-of-constants} to fractional orders, and is 
therefore taken as the definition of the \emph{mild solution} 
to~\eqref{eq:natural-Dirichlet-abstract}.

\subsection{Contributions}
\label{sec:prelims:contr}

The main contribution of this work is the 
introduction and motivation of~\eqref{eq:mildsol-Dirichlet} as the
definition of the mild solution to~\eqref{eq:natural-Dirichlet-abstract}
for $s \in (0,\infty) \setminus \N$ and bounded continuous $g$, rigorously 
formulated in Definition~\ref{def:mildsol-Dirichlet}.
This definition is motivated by Theorem~\ref{thm:Lp-implies-mild}, which 
shows that \emph{$L^p$-solutions} to~\eqref{eq:natural-Dirichlet-abstract} 
satisfy~\eqref{eq:mildsol-Dirichlet} under certain natural conditions.
Although its proof relies on the uniform exponential 
stability of 
$(S(t))_{t\ge0}$, the resulting formula is well-defined under the more general 
assumption 
that $(S(t))_{t\ge0}$ is uniformly bounded. In particular, this includes the 
case 
$A = 0$, 
meaning that~\eqref{eq:mildsol-Dirichlet} with $S(\,\cdot\,) \equiv \id_X$ 
can be 
viewed as 
a solution to the 
Dirichlet problem associated to the fractional time derivative 
$\partial_t^s$. 
Likewise, if 
$(S(t))_{t\ge0}$ is uniformly exponentially stable, then the integral 
in~\eqref{eq:mildsol-Dirichlet} also converges for $\{s\} = 0$, so 
that~\eqref{eq:mildsol-Dirichlet} remains meaningful for integers 
$s = n \in 
\N$ and reduces to the integer-order solution formula:
\begin{equation}
	u(t) = \sum_{k=0}^{n-1}
	\frac{(t-t_0)^{k}}{k!}
	S(t-t_0) 
	[(\partial_t + A)^{k} g](t_0), 
	\quad 
	t \in (t_0, \infty).
\end{equation}
If $(S(t))_{t\ge0}$ is merely uniformly bounded, then we can still show that 
the 
first term 
of~\eqref{eq:mildsol-Dirichlet} converges to $S(t - t_0)g(t_0)$
as $\{s\} \uparrow 1$ for all 
$t \in (t_0, \infty)$, see Proposition~\ref{prop:reduction}.
For constant initial data $g \equiv x \in X$, we find 
that~\eqref{eq:mildsol-Dirichlet} can be conveniently expressed in terms of an 
operator-valued incomplete 
gamma 
function, see Corollary~\ref{cor:g-equiv-x}.

In addition to~\eqref{eq:mildsol-Dirichlet}, we define solution concepts for the
Cauchy problems associated to fractional parabolic 
Riemann--Liouville and Caputo type derivative operators
(see Proposition~\ref{prop:Int-Deriv-gamma-RL} and Definitions~\ref{def:RL-ACP} 
and~\ref{def:C-ACP}) for comparison.
The 
higher-order terms comprising the summation in~\eqref{eq:mildsol-Dirichlet}
turn out to 
be analogous to the corresponding terms in the Riemann--Liouville solution.
The integral term in~\eqref{eq:mildsol-Dirichlet}, however, is continuous at 
$t_0$ under mild 
conditions on $(S(t))_{t\ge0}$ or $g$, in contrast to the lowest-order term
in the Riemann--Liouville formula, which has a singularity there. As opposed to 
the Caputo type initial value problem, the solutions 
to~\eqref{eq:natural-Dirichlet-abstract} 
are in general different for distinct $s_1, s_2 \in (n, n+1)$ for $n \in \N_0$.

To the best of the author's knowledge, the solution formula given 
by~\eqref{eq:mildsol-Dirichlet} is
new even in the scalar-valued case $X \coloneqq \C$,
$A \coloneqq a \in \clos \C_+$ and
$(S(t))_{t\ge0} = (e^{-at})_{t\ge0}$, as are the Riemann--Liouville and 
Caputo type solutions for $a \in \clos\C_+ \setminus \{0\}$.

\subsection{Outline}
\label{section:intro:outline}
This article is structured as follows. In Section~\ref{section:prelims},
we establish some notation and collect preliminary results regarding 
semigroups, fractional calculus, first-order abstract Cauchy problems and
the Phillips functional calculus associated to semigroup generators. 
These notions are first used in Section~\ref{section:frac-prob-R} to 
investigate 
problem~\eqref{eq:natural-Dirichlet-abstract} for $t_0 = -\infty$, i.e., in the 
absence of prescribed initial data.
Section~\ref{sec:Dirichlet-frac-parab} is concerned with the 
rigorous definition of mild and $L^p$-solutions 
to~\eqref{eq:natural-Dirichlet-abstract}; after establishing the relation 
between these two concepts, we focus on the mild solution and establish some of 
its most important properties. The comparison with the solution 
concepts 
associated to 
Riemann--Liouville and 
Caputo type initial value problems is presented in 
Section~\ref{section:RL-C-ACP}.

\section{Preliminaries}
\label{section:prelims}

\subsection{Notation}

Let $\N \coloneqq \{1, 2, 3, \dots \}$ and $\N_0 \coloneqq \N \cup \{0\}$ denote
the sets of positive and non-negative integers, respectively. 
We write $\lfloor \,\cdot\, \rfloor$ and $\lceil \,\cdot\, \rceil$ for the 
floor 
and ceiling functions; the fractional part of $\alpha \in [0,\infty)$ is 
defined by $\{\alpha\} \coloneqq \alpha - \lfloor \alpha \rfloor$. 
The maximum (respectively, minimum) of two real numbers $\alpha,\beta\in\R$ is 
denoted 
by $\alpha \vee \beta$ (respectively, $\alpha \wedge \beta$).
The function $t \mapsto t^\beta_+$ is defined by $t^\beta_+ \coloneqq 
t^\beta$ 
if $t \in (0,\infty)$ and $t^\beta_+ \coloneqq 0$ otherwise.
The real and 
imaginary parts of a complex number $z \in \C$ are respectively denoted by 
$\Re{z}$ and $\Im{z}$; if $z \neq 0$, then its argument is written as $\arg z 
\in (-\pi, \pi]$. The open and closed right half-planes of the complex 
plane are denoted by 
\[
\C_+ \coloneqq \{ z \in \C : \Re z > 0\}
\quad 
\text{and} 
\quad 
\clos \C_+ \coloneqq \{ z \in \C : \Re z \ge 0\},
\]
respectively.
The identity map on a set $B$ is denoted by $\id_B \from B
\to B$ and we write $\mathbf 1_{B_0} \from B \to \{0,1\}$ for the indicator
function of a subset $B_0 \subseteq B$.

Throughout this article, $(X, \norm{\,\cdot\,}{X})$ denotes a Banach
space over the complex scalar field $\C$; the real case can be treated using 
complexifications.
The Banach space of bounded linear operators from $X$ to a Banach space $Y$ 
with the uniform operator
norm is denoted by $\LO(X; Y)$; for $X = Y$ we set $\LO(X) \coloneqq \LO(X; X)$.
The notation $A \from \dom{A} \subseteq X
\to X$ indicates that $A$ is a (possibly unbounded) linear operator on $X$ with
domain $\dom{A}$ and graph
$\mathsf G(A) \coloneqq \{(x, Ax) : x \in \dom{A}\}$. 

Let $(S, \mathscr{A} \!, \mu)$ be a measure space. For $f \from S \to \R$ 
and $x 
\in X$, define ${f \otimes x \from S \to X}$ by $[f \otimes x](s) 
\coloneqq 
f(s)x$. A function $f \from S \to X$ is said to be strongly $\mu$-measurable if 
it is the $\mu$-almost everywhere (``a.e.'')\ limit of 
$\mu$-simple functions, i.e., linear combinations of
$\mathbf 1_B \otimes x$ with $B \in \mathscr{A}, \mu(B) < \infty$ and $x \in 
X$. For 
$p \in 
[1,\infty]$, let $L^p(S; X)$ denote the Bochner space of (equivalence classes 
of) $p$-integrable functions with norm 
\[
\norm{f}{L^p(S; X)} \coloneqq
\begin{cases}
	(\int_S \norm{f(s)}{X}^p \rd 
	\mu(s))^{\nicefrac{1}{p}}, &\text{if } p \in [1,\infty); \\
	\operatorname{ess\,sup}_{s \in 
		S}\norm{f(s)}{X}, &\text{if } p = \infty.
\end{cases}
\]
Intervals $J \subseteq \R$ are 
equipped with the Lebesgue $\sigma$-algebra and measure. The Banach space of 
bounded 
continuous functions $u \from J \to X$, endowed with the supremum norm, is 
denoted 
by 
$(\Cb(J; X), \norm{\,\cdot\,}{\infty})$.

\subsection{Strongly measurable semigroups of bounded linear operators}
\label{sec:prelims:semigroups}

A family $(S(t))_{t\ge0}$ of bounded linear operators on a complex Banach 
space 
$X$ is said to be a \emph{semigroup} if $S(0) = \id_X$ and, 
for all $t, s \ge 0$, we have $S(t + s) = S(t)S(s)$. It is called a 
\emph{strongly measurable semigroup} if, in addition, the orbit 
$t \mapsto S(t)x$ of any $x \in X$ is a strongly measurable mapping from 
$[0,\infty)$ to $X$. For a strongly measurable semigroup, it holds that
$t \mapsto 
S(t)x$ is 
continuous on $(0, \infty)$ for every $x \in X$, with continuity at zero 
if and only if $x \in \clos{\dom{A}}$. In what follows, we exclusively 
consider \emph{locally 
	bounded} strongly measurable semigroups, which satisfy
\begin{equation}\label{eq:semigroup-est}
	\exists M_0 \in [1,\infty), \, w \in \R : 
	\quad 
	\norm{S(t)}{\LO(X)} 
	\le 
	M_0 e^{-wt}, \quad \forall t \in [0,\infty).
\end{equation}

Analogously to the relation between a matrix $G$ and its corresponding 
matrix 
exponential function $S(t) \coloneqq e^{tG}$, we can more generally associate 
an \emph{infinitesimal generator} $G \from \dom{G} \subseteq X \to X$ to any
locally bounded strongly measurable semigroup $(S(t))_{t\ge0}$. It is defined 
by the following 
property: For all $\lambda \in \C$ such that $\Re \lambda > w$, the operator
$\lambda \id_X-G$ admits a bounded inverse, and we have
\[
\scalar{(\lambda \id_X - G)^{-1}x, x^*}{}
=
\int_0^\infty e^{-\lambda t} \scalar{S(t)x, x^*}{}\rd t
\quad 
\text{for all } x \in X \text{ and } x^* \in X^*,
\]
where $\scalar{\,\cdot\, , \,\cdot\,}{}$ denotes the duality pairing 
between 
$X$ and its dual $X^*$.
We refer the reader to~\cite[Appendix~K]{HvNVWVolumeIII}
for a more detailed summary of the theory of (strongly) measurable semigroups.

The above notions will be applied to the operator $A$ 
from~\eqref{eq:natural-Dirichlet-abstract} on $X$, on which we impose the 
following standing assumptions:
\begin{assumption}\label{ass:bdd-semigroup}
	Let $-A \from \dom{A} \subseteq X \to X$ be the infinitesimal generator of 
	a locally bounded strongly measurable semigroup 
	$(S(t))_{t\ge0} \subseteq \LO(X)$, which 
	satisfies~\eqref{eq:semigroup-est}.
	More precisely, we suppose that $(S(t))_{t \ge 0}$ is either
	\begin{enumerate}[leftmargin=1cm, label=(\roman*)]
		\item \label{ass:bdd-semigroup:bdd}
		\emph{uniformly bounded}, meaning that $w \in [0,\infty)$, or 
		\item \label{ass:bdd-semigroup:exp}
		\emph{uniformly exponentially stable}, meaning that 
		$w \in (0,\infty)$.
	\end{enumerate}
	We may sometimes additionally assume that $(S(t))_{t\ge0}$ is
	\begin{enumerate}[leftmargin=1cm, label=(\roman*)]
		\setcounter{enumi}{2}
		\item \label{ass:bdd-semigroup:ana}
		\emph{bounded analytic}, i.e.,
		$(0,\infty) \ni t \mapsto S(t) \in \LO(X)$ admits a bounded holomorphic 
		extension to
		$\Sigma_\varphi \coloneqq \{ z \in \C : \lvert\arg z\rvert < \varphi \}$
		for some $\varphi \in (0, \frac{1}{2}\pi)$.
	\end{enumerate}
\end{assumption}
%

\subsection{Solution concepts for first-order abstract Cauchy problems}
\label{sec:11st-order-ACP}
In this section we define and relate the concepts of strong, mild and 
$L^p$-solutions to the first-order Cauchy problem~\eqref{eq:ACP} on an interval 
$J = (t_0, \infty)$, where $t_0 \in [-\infty, \infty)$. For a more 
detailed reference, see~\cite[Chapter~17]{HvNVWVolumeIII}.

Let $\Lloc^1(\clos J; X)$ denote the space of strongly measurable functions 
from $\clos J$ to $X$ which are integrable on every compact subset of $\clos J$.
For $b < a$ we set $\int_a^b \coloneqq -\int_b^a$.

\begin{definition}[Strong solution]\label{def:ACP-strong-sol}
	A strongly measurable function $u \from J \to X$ is said to be a 
	\emph{strong solution} of~\eqref{eq:ACP} associated with 
	$f \in \Lloc^1(\overline J; X)$ if 
	\begin{enumerate}[(i)]
		\item $u(t) \in \dom{A}$ for almost all $t \in J$ and 
		$t \mapsto Au(t) \in
		\Lloc^1(\overline J; X)$;
		\item for almost all $t \in J$ we have
		\begin{equation}\label{eq:def-ACP-strong-sol}
			\begin{aligned}
				u(t) + \int_{t_0}^t Au(\tau) \rd \tau &= x
				+
				\int_{t_0}^t f(\tau)\rd \tau 
				\quad
				&&\text{if } t_0 \in \R; \\
				u(t) + \int_{0}^t Au(\tau) \rd \tau &= 
				u(0) + \int_{0}^t f(\tau)\rd \tau 
				\quad
				&&\text{if } t_0 = -\infty.
			\end{aligned}
		\end{equation}
	\end{enumerate}
\end{definition}
As antiderivatives of locally integrable functions are continuous,
see~\cite[Proposition~2.5.9]{HvNVWVolumeI}, it follows that any strong solution
admits
a continuous representative, so that the pointwise evaluation of $u$
in~\eqref{eq:def-ACP-strong-sol} is meaningful. In fact, identifying $u$ with
this representative, it turns out to be (classically) differentiable for 
almost
all $t \in J$, where it holds that 
$
u'(t) + Au(t) = f(t),
$
see~\cite[Equation~(17.3)]{HvNVWVolumeIII}. This implies $u' = f - Au 
\in 
\Lloc^1(\clos J; X)$, hence $u$ is weakly differentiable and its 
weak derivative $\partial_t u$ coincides a.e.\ with $u'$, 
again by~\cite[Proposition~2.5.9]{HvNVWVolumeI}.

Next, we turn to the definition of a mild solution to~\eqref{eq:ACP}. We 
will 
see in Proposition~\ref{prop:ACP-strong-vs-mild} below that mild solutions 
are 
weaker than strong solutions.

\begin{definition}[Mild solution]\label{def:ACP-mild-sol} 
	Suppose that Assumption~\ref{ass:bdd-semigroup}\ref{ass:bdd-semigroup:exp} 
	is satisfied
	and let $J \coloneqq (t_0, \infty)$ for a given $t_0 \in [-\infty, \infty)$.
	The \emph{mild solution} to~\eqref{eq:ACP} with
	$f \in L^p(J; X)$ for $p \in [1,\infty]$
	is the
	function $u \in \Cb(\clos J; X)$ defined 
	for all $t \in \clos{J}$ by 
	\begin{equation}\label{eq:def-ACP-mild-sol}
		\begin{aligned}
			u(t)
			&\coloneqq 
			S(t - t_0)x 
			+ 
			\int_{t_0}^t S(t - \tau) f(\tau) \rd \tau
			&&\text{ if } t_0 \in \R \text{ and } x \in \clos{\dom{A}}; \\ 
			u(t)
			&\coloneqq 
			\int_{-\infty}^t S(t-\tau) f(\tau) \rd \tau
			&&\text{ if } t_0 = -\infty.
		\end{aligned}
	\end{equation}
\end{definition}
The continuity of the mild solution $u$ defined by~\eqref{eq:def-ACP-mild-sol}
follows from~\cite[Proposition~K.1.5(3)]{HvNVWVolumeIII} and
Proposition~\ref{prop:frac-par-int-props}\ref{prop:frac-par-int-props:Lp-Cub}. 
The following is a slight extension 
of~\cite[Proposition~17.1.3]{HvNVWVolumeIII} for the class of time intervals 
considered in this work.

\begin{proposition}\label{prop:ACP-strong-vs-mild}
	Suppose that Assumption~\ref{ass:bdd-semigroup}\ref{ass:bdd-semigroup:exp} 
	holds.
	Let $J \coloneqq (t_0, \infty)$ for a given $t_0 \in [-\infty, 
	\infty)$,
	$f \in L^p(J; X)$ for some ${p \in [1,\infty]}$ and $x \in \clos{\dom{A}}$
	if $t_0 \in \R$.
	Then, for every $u \in \Cub(\clos{J}; X)$, the following assertions are 
	equivalent:
	\begin{enumerate}[(a)]
		\item $u$ is a strong solution of~\eqref{eq:ACP} in the sense of 
		Definition~\ref{def:ACP-strong-sol};
		\label{prop:ACP-strong-vs-mild:a}
		\item $u$ is the mild solution of~\eqref{eq:ACP} in the sense of 
		Definition~\ref{def:ACP-mild-sol} and $u$ is (classically)
		differentiable almost everywhere with $u' \in \Lloc^1(\clos{J}; 
		X)$; 
		\label{prop:ACP-strong-vs-mild:b}
		\item $u$ is the mild solution of~\eqref{eq:ACP} in the sense of 
		Definition~\ref{def:ACP-mild-sol}, $u(t) \in \dom{A}$ for almost
		all $t \in J$ and $t \mapsto Au(t) \in \Lloc^1(\clos{J}; X)$.
		\label{prop:ACP-strong-vs-mild:c}
	\end{enumerate}
\end{proposition}
\begin{proof}
	If $t_0 \in \R$, then the required modifications of the proof
	of~\cite[Proposition~17.1.3]{HvNVWVolumeIII} are straightforward. We 
	briefly comment	on the case $t_0 = -\infty$.
	Let $\lambda \in \C$ be such that $\lambda \id_X + A$ admits a bounded
	inverse. 
	
	\ref{prop:ACP-strong-vs-mild:a}$\implies$\ref{prop:ACP-strong-vs-mild:b}:
	For $u$ as in Definition~\ref{def:ACP-strong-sol} and $t \in \R$, we define
	$v \from (-\infty, t] \to X$
	by 
	\[
	v(\tau) \coloneqq (\lambda \id_X + A)^{-1} S(t-\tau) u(\tau), \quad 
	\tau \in (-\infty, t].
	\]
	Fixing $t' < t$, and arguing as in the original proof---except for 
	integrating over 
	$(t', t)$ instead of $(0, t)$---we find
	\[
	u(t) = S(t-t') u(t') + \int_{t'}^t S(t-\tau) f(\tau) \rd \tau.
	\]
	As $t' \to -\infty$,
	the first term vanishes by $u \in \Cb(\R; X)$ and 
	Assumption~\ref{ass:bdd-semigroup}\ref{ass:bdd-semigroup:exp}, and the 
	second term 
	converges to $\int_{-\infty}^t S(t-\tau) f(\tau) \rd \tau$ by 
	dominated convergence.
	
	\ref{prop:ACP-strong-vs-mild:b}$\implies$\ref{prop:ACP-strong-vs-mild:a}
	and 
	\ref{prop:ACP-strong-vs-mild:c}$\implies$\ref{prop:ACP-strong-vs-mild:a}:
	Use the following analog
	to~\cite[Equation~(17.4)]{HvNVWVolumeIII}: 
	\begin{equation}
		\begin{aligned}
			\int_{0}^t 
			A(\lambda \id_X + A)^{-1} u(\tau)
			\rd \tau 
			=
			-(\lambda \id_X + A)^{-1} \biggl[
			u(t) - u(0)
			-
			\int_{0}^t f(\tau) \rd \tau
			\biggr].
		\end{aligned}
		\tag*{\qedhere}
	\end{equation}
\end{proof}
Finally, we introduce a stronger notion of solutions, which will later be 
used as the basis for an analogous solution concept for the fractional problem:
\begin{definition}[$L^p$-solution]\label{def:Lpsol-firstorderACP}
	Let $J \coloneqq (t_0, \infty)$ for $t_0 \in [-\infty, 
	\infty)$ and
	$f \in L^p(J; X)$ with $p \in [1,\infty]$. A strong solution $u$
	to~\eqref{eq:ACP} (in the sense of Definition~\ref{def:ACP-strong-sol}) is 
	said to be an \emph{$L^p$-solution} if $t \mapsto Au(t) \in L^p(J; X)$.
\end{definition}
%

\subsection{Fractional parabolic calculus}
\label{section:prelims:frac-par-calc}

Fractional calculus classically refers to the idea of generalizing the 
basic 
operations of calculus, namely differentiation and integration, beyond the 
integer-order cases (which simply consist of 
repeated application). For example, a fractional 
derivative $\partial_t^\alpha f$ of a sufficiently smooth given function 
$f \from \R \to \R$ should be defined for non-integer 
$\alpha \in (0,\infty)$, in such a way that
$\partial_t^{1/2} \partial_t^{1/2} f = \partial_t f$.
To this end, a common first object of study is the \emph{Riemann--Liouville 
	fractional 
	integral}, defined for suitable $f$ by 
\begin{equation}
	I^s_{\text{RL}} f(t) \coloneqq \frac{1}{\Gamma(s)}\int_{-\infty}^t 
	(t-\tau)^{s-1} 
	f(\tau) \rd \tau, 
	\quad 
	s \in (0,\infty), \quad
	t \in \R,
\end{equation}
Since this is a fractional-order generalization of 
the integration operator $I^1_{\text{RL}}$, which can itself be viewed as an 
inverse of the 
derivative $\partial_t$, we can interpret $I^s_{\text{RL}}$ as 
$\partial_t^{-s}$. This leads us to define the following fractional 
derivatives:
\begin{equation}
	[\partial_t^s]_{\text{RL}} \coloneqq \partial_t^{\lceil s \rceil} 
	I^{\lceil s \rceil - 
		s}
	\quad 
	\text{and} 
	\quad 
	[\partial_t^s]_{\text{C}} \coloneqq 
	I^{\lceil s \rceil - s}\partial_t^{\lceil s \rceil},
\end{equation}
which are known as the Riemann--Liouville and Caputo fractional 
(time) 
derivatives, respectively. For a comprehensive overview of these and other 
fractional time derivatives, as well as their application in fractional 
differential equations, we refer to the monographs~\cite{KST2006, SKM1993, 
	Podlubny1999}.

Under the assumption that $-A$ generates a semigroup $(S(t))_{t\ge0}$ 
satisfying Assumption~\ref{ass:bdd-semigroup}\ref{ass:bdd-semigroup:exp}, we 
can extend these notions to define the lesser-known fractional 
\emph{parabolic} 
integration and
differentiation operators $\fI^s$ and $\fD^s$, which are the 
rigorous
definitions of the expressions $(\partial_t + A)^{-s}$ and $(\partial_t + A)^s$
from Section~\ref{section:intro}, respectively.
For $s \in (0,\infty)$, let the integration kernel $k_s \from \R \to 
\LO(X)$
be defined 
by $k_s(\tau) \coloneqq \frac{1}{\Gamma(s)} \tau^{s-1}_+ S(\tau)$
for all $\tau \in \R$.
Given $u \from \R \to X$,
its \emph{Riemann--Liouville type fractional parabolic integral} 
$\mathfrak I^s u \from \R \to X$ \emph{of order $s$} is defined 
by
\begin{equation}\label{eq:def-Igamma}
	\mathfrak I^s u(t) 
	\coloneqq 
	k_s * u(t)
	\coloneqq
	\frac{1}{\Gamma(s)} 
	\int_0^\infty \tau^{s-1} S(\tau) u(t - \tau) \rd \tau
\end{equation}
whenever this Bochner integral
converges for a.e.\ $t \in \R$.
For $s = 0$ we set $\mathfrak I^0 \coloneqq \id_X$. 
When viewed as a linear operator, the fractional parabolic integral 
$\fI^s$ turns out to map its domain $L^p(\R;X)$ boundedly to itself for every 
$s \in [0,\infty)$. In fact, we have the following properties of $(\mathfrak 
I^s)_{s\in[0,\infty)}$, which will be used 
throughout this work:

\begin{proposition}\label{prop:frac-par-int-props}
	Suppose that 
	Assumption~\ref{ass:bdd-semigroup}\ref{ass:bdd-semigroup:exp} 
	holds. Let $p \in [1,\infty]$ and $s \in [0,\infty)$.
	The following assertions hold:
	\begin{enumerate}[leftmargin=1cm, label=(\alph*)]
		\item \label{prop:frac-par-int-props:Lp-bdd}
		$\mathfrak I^s \in \LO(L^p(\R;
		X))$ with 
		$\norm{\mathfrak I^s}{\mathscr L(L^p(\R;X))}
		\le
		\frac{M_0}{w^{s}}$.
		\item \label{prop:frac-par-int-props:Lp-Cub}
		$\mathfrak I^s \in \LO(L^p(\R; X); \Cb(\R; X))$ 
		if
		\begin{equation}\label{eq:sp-Ds-cont}
			\begin{cases}
				p = 1 &\text{and} \quad\: s \in [1,\infty), \quad \text{ or}, \\
				p \in (1,\infty) &\text{and} \quad\: s \in 
				(\nicefrac{1}{p},\infty).
			\end{cases}
		\end{equation}
		\item \label{prop:frac-par-int-props:semigroup}
		$
		\mathfrak I^{s_1} \mathfrak I^{s_2} u
		= 
		\mathfrak I^{s_1+s_2} u
		$ a.e.\ 
		for all $s_1, s_2 \in [0,\infty)$ and $u \in L^p(\R; X)$.
		\item \label{prop:frac-par-int-props:invariant-funcs}
		Given $x \in X$, if $p \in [1, \frac{1}{1-s_1})$
		and $s_1 \in (0,1)$ or $p \in [1, \infty]$ and $s_1 \in [1,\infty)$, 
		then
		$k_{s_1} \otimes x \in L^p(\R; X)$
		and $\mathfrak I^{s_2} (k_{s_1} \otimes x) = k_{s_1 + s_2} \otimes x$
		for all $s_2 \in [0,\infty)$.
	\end{enumerate}
\end{proposition}
\begin{proof}
	Estimate~\eqref{eq:semigroup-est} implies 
	$\norm{k_s}{L^1(\R; \LO(X))} \le M_0 w^{-s}$ for all $s \in (0,\infty)$, 
	so that 
	Minkowski's integral inequality~\cite[Section~A.1]{Stein1970}
	yields~\ref{prop:frac-par-int-props:Lp-bdd}.
	
	If $p' \in [1,\infty]$ is such that $\frac{1}{p} + \frac{1}{p'} = 1$
	and $u \in L^p(\R; X)$,
	then ${k_s \in L^{p'}(\R; \LO(X))}$ for $s$ as in the statement
	of~\ref{prop:frac-par-int-props:Lp-Cub}, and the result follows from 
	H\"older's 
	inequality and the continuity of translations in $L^q(\R; X)$ or 
	$L^q(\R;\LO(X))$ for $q \in [1, \infty)$.
	
	Assertions~\ref{prop:frac-par-int-props:semigroup} 
	and~\ref{prop:frac-par-int-props:invariant-funcs} follow by combining 
	the semigroup property of $(S(t))_{t\ge0}$,
	Fubini's theorem and \cite[Equation~(5.12.1)]{Olver2010}.
\end{proof}

For $p \in [1,\infty]$,
let $W^{1,p}(\R; X)$ denote the Bochner--Sobolev space consisting of functions 
$u \in L^p(\R; X)$ whose weak derivative $\partial_t u$ also belongs to
$L^p(\R; X)$. Identifying $A \from \dom{A} \subseteq X \to X$ with the 
operator 
\[
\cA \from L^p(\R;\mathsf D(A)) \subseteq L^p(\R;X) \to L^p(\R;X)
\]
defined by 
$[\cA u](\,\cdot\,) \coloneqq Au(\,\cdot\,)$, we can view 
$\partial_t + A$ as an operator on $L^p(\R;X)$ with domain 
$L^p(\R; \dom{A}) \cap W^{1,p}(\R; X)$.
In conjunction with the operators $(\mathfrak I^s)_{s\ge0}$ 
from~\eqref{eq:def-Igamma}, this leads to the definition of the
\emph{Riemann--Liouville type fractional parabolic derivative of order $s \in 
	[0,\infty)$}:
\begin{equation}\label{eq:def-RL-fracpar-deriv-R}
	\begin{aligned}
		\mathfrak D^s u &\coloneqq (\partial_t + A)^{\lceil s \rceil} 
		\mathfrak I^{\lceil s \rceil-s} u, 
		\\
		u \in \mathsf D(\mathfrak D^s) 
		&\coloneqq
		\{ u \in L^p(\R; X) 
		: 
		\mathfrak I^{\lceil s \rceil-s} u 
		\in 
		\dom{(\partial_t + A)^{\lceil s \rceil}} \}.
	\end{aligned}
\end{equation}
Note that we do not explicitly indicate the dependence of $\fD^s$ and $\fI^s$ 
on 
$p \in [1,\infty]$ in the notation, instead leaving it to be inferred from 
context. 
\begin{remark}
	While the terminology 
	``fractional parabolic'' is inspired by the case $A = -\Delta$ acting on a 
	function space such as $X = L^2(\cD)$, our 
	setting is considerably more general. 
\end{remark}
\begin{remark}\label{rem:motivation-Ds}
	Let us briefly elaborate on our choice for the 
	Riemann--Liouville type 
	operators $\fD^s$
	and $\fI^s$ as the rigorous interpretations of $(\partial_t + A)^s$
	and $(\partial_t + A)^{-s}$, respectively. Under some additional 
	assumptions on 
	$A$ 
	and $X$, it can be shown that the 
	sum operator
	$\partial_t + \cA$ admits an extension $\cB$ which is 
	\emph{sectorial}~\cite[Section~16.3]{HvNVWVolumeIII}. 
	For this class of operators, fractional powers
	can be defined as in~\cite[Section~15.2]{HvNVWVolumeIII}. 
	If, for instance, $-A$ generates an exponentially bounded \emph{strongly 
		continuous} semigroup 
	on a 
	Hilbert space $X = H$, then it holds that $\cB^{-s} = \mathfrak I^s$ for 
	all $s \in [0,\infty)$, cf.~\cite[Proposition~3.2 and 
	Equation~(3.9)]{KW2023}.
	Then~\cite[Proposition~15.1.12(2)]{HvNVWVolumeIII} implies that
	\begin{equation}
		\cB^s = \cB^{\lceil s \rceil} \cB^{s - \lceil s \rceil}
		\quad \text{with} \quad 
		\dom{\cB^s} = \{ u \in L^p(\R; X) : u \in \dom{\cB^{\lceil s \rceil 
		}}\}.
	\end{equation}
	Hence in this situation we find that 
	$\cB^s$ is an extension of $\fD^s$. In particular, if we have $\cB = 
	\partial_t + 
	\cA$ (a property which is closely related to the maximal $L^p$-regularity 
	of 
	$A$, see~\cite[Proposition~17.3.14]{HvNVWVolumeIII}), then in fact 
	$\cB^s = \fD^s$.
	
	Thus, there is a close relation between the Riemann--Liouville 
	fractional 
	parabolic operators and fractional powers of sectorial extensions of 
	$\partial_t + \cA$. We choose the former viewpoint for the sake of 
	simplicity, generality and consistency with the analogous definitions 
	for 
	$A = 0$ in fractional calculus texts such as~\cite{SKM1993, KST2006, 
		Podlubny1999, GalWarma2020}.
\end{remark}
%

\subsection{Phillips functional calculus}
\label{section:prelims:flips}
If Assumption~\ref{ass:bdd-semigroup}\ref{ass:bdd-semigroup:bdd} holds and
$f\from \clos \C_+ \to \C$
can be written as the Laplace transform of a complex 
Borel measure $\mu$ of bounded variation on $[0,\infty)$, i.e., if we have
\begin{equation}
	f(z) = \cL[\mu](z) := \int_{[0,\infty)} e^{-zs} \rd \mu(s)
\end{equation}
for all 
$z \in \clos \C_+$, then we define the operator
\begin{equation}
	f(A) 
	\coloneqq 
	\biggl[\int_{[0,\infty)} e^{-sz}\rd\mu(s)\biggr](A)  
	\coloneqq 
	\int_{[0,\infty)} S(s) \rd \mu(s) \in \LO(X).
\end{equation}
The map $f \mapsto f(A)$, called
the \emph{Phillips functional calculus} for $A$, 
is an algebra homomorphism from the space of Laplace transforms to $\mathscr 
L(X)$, see~\cite[Remark~3.3.3]{Haase2006}.

Note that $S(t) 
= 
\cL[\delta_t](A)
=
(e^{-zt})(A)$,
where $\delta_t$ denotes the Dirac measure concentrated at $t \in [0,\infty)$. 
Moreover, for any $\alpha, \varepsilon \in (0,\infty)$ we can define the 
negative 
fractional powers of $A + \varepsilon \id_X$ by 
\begin{equation}
	(A + \varepsilon \id_X)^{-\alpha} 
	\coloneqq [(z + \varepsilon)^{-\alpha}](A)
	=
	\cL\biggl[\frac{s^{\alpha-1}e^{-\varepsilon s}}{\Gamma(\alpha)} \rd 
	s\biggr](A),
\end{equation}
see~\cite[Proposition~3.3.5]{Haase2006}. 
These can be used to define $(A + \varepsilon \id_X)^\alpha$ and, in turn,
$A^\alpha \from \dom{A^\alpha} \subseteq X \to X$ in a manner which is 
consistent with other common definitions of fractional powers, 
cf.~\cite[Propositions~3.1.9 and~3.3.2]{Haase2006}. Under 
Assumption~\ref{ass:bdd-semigroup}\ref{ass:bdd-semigroup:exp}, we can also 
allow for $\varepsilon = 0$ directly in the above to define 
(see~\cite[Corollary~3.3.6]{Haase2006}):
\begin{equation}\label{eq:neg-frac-power}
	A^{-\alpha} 
	\coloneqq 
	\frac{1}{\Gamma(\alpha)} \int_0^\infty \sigma^{\alpha - 1} S(\sigma) 
	\rd \sigma \in 
	\LO(X).
\end{equation}
%

\section{Fractional-order inhomogeneous abstract Cauchy problem on \texorpdfstring{$\mathbb{R}$}{R}}
\label{section:frac-prob-R}

In this section, we consider the inhomogeneous
abstract Cauchy problem associated to the fractional operator $\fD^s$ and $f 
\in L^p(\R; X)$, where $s 
\in (0,\infty)$, $p \in [1,\infty]$ and $J = \R$:
\begin{equation}\label{eq:frac-ACP-R}
	\mathfrak D^s u(t) = f(t), \quad t \in \R.
\end{equation}
Note that we do not impose any initial data here since the problem is 
posed on 
the 
entire real line. The solution concepts for~\eqref{eq:frac-ACP-R} which we 
will define are the following 
fractional-order analogs to the notion of $L^p$-solutions and mild solutions 
(see Definitions~\ref{def:Lpsol-firstorderACP} 
and~\ref{def:ACP-mild-sol}, respectively):
\begin{definition}[$L^p$-solution]\label{def:Lpsol-ACP-R}
	Suppose that 
	Assumption~\ref{ass:bdd-semigroup}\ref{ass:bdd-semigroup:exp} 
	is satisfied. 
	Let $s \in (0,\infty)$, $p \in [1,\infty]$ and 
	$f \in L^p(\R; X)$. Then $u \in L^p(\R; X)$ is called an 
	\emph{$L^p$-solution} 
	to~\eqref{eq:frac-ACP-R} if $u \in \dom{\fD^s}$ and equation~\eqref{eq:frac-ACP-R} holds almost everywhere on $\R$.
\end{definition}
It is a consequence of Proposition~\ref{prop:ID-DI-s}\ref{prop:ID-DI-s:ID} 
below that the $L^p$-solution to~\eqref{eq:frac-ACP-R} is unique if it exists. 
The question of existence of the $L^p$-solution for all $f \in L^p(\R; 
X)$ is highly nontrivial; in the case $s = 1$, it characterizes a deep property 
of 
the linear operator $A$ called \emph{maximal $L^p$-regularity}. Since the 
present article is primarily concerned with the concept of mild solutions, to 
be defined next, we shall not investigate this matter further, and instead 
refer 
to~\cite[Chapter~17]{HvNVWVolumeIII} for a more detailed account of the 
topic of maximal $L^p$-regularity.
\begin{definition}[Mild solution]\label{def:mildsol-ACP-R}
	Suppose that 
	Assumption~\ref{ass:bdd-semigroup}\ref{ass:bdd-semigroup:exp} 
	is satisfied. Let $s \in (0,\infty)$ and $p \in [1,\infty]$ 
	satisfy~\eqref{eq:sp-Ds-cont}. The 
	\emph{mild solution} to~\eqref{eq:frac-ACP-R} with
	$f \in L^p(\R; X)$
	is the
	function $u \in \Cb(\R; X)$ defined 
	for all $t \in \R$ by
	\begin{equation}
		u(t)
		\coloneqq 
		\fI^s f(t) =
		\frac{1}{\Gamma(s)}
		\int_{-\infty}^t (t-\tau)^{s - 1} S(t-\tau) f(\tau) \rd \tau.
	\end{equation}
\end{definition}
The mild solution exists and is unique by definition, since it is given by 
an 
explicit formula. Moreover, in view of 
Proposition~\ref{prop:frac-par-int-props}\ref{prop:frac-par-int-props:Lp-Cub}, 
it is indeed continuous under the given assumptions on $s$ and $p$.

The next proposition shows that the fractional parabolic
derivative and integral are inverse to each other whenever the respective 
left-hand
sides are well-defined. In particular, it implies that $L^p$-solutions are 
mild 
solutions whenever the parameters $s$ and $p$ are such 
that~\eqref{eq:sp-Ds-cont} holds, see Corollary~\ref{cor:Lp-implies-mild-ACP-R} 
below.

\begin{proposition}\label{prop:ID-DI-s}
	Suppose that Assumption~\ref{ass:bdd-semigroup}\ref{ass:bdd-semigroup:exp} 
	holds. Let $s \in [0,\infty)$, ${p \in [1,\infty]}$ and 
	$u \in L^p(\R; X)$. Then the following assertions hold:
	\begin{enumerate}[leftmargin=1cm, label=(\alph*)]
		\item \label{prop:ID-DI-s:DI} 
		If $\mathfrak I^s u \in \dom{\mathfrak D^s}$, then 
		$\mathfrak D^s \mathfrak I^s u = u$ a.e.
		\item \label{prop:ID-DI-s:ID} 
		If $u \in \dom{\mathfrak D^s}$, then 
		$\mathfrak I^s \mathfrak D^s u = u$ a.e.
	\end{enumerate}
\end{proposition}
\begin{proof}
	\ref{prop:ID-DI-s:DI} For $s = 1$, $v \coloneqq \mathfrak 
	I^1 u$ 
	is the mild
	solution to~\eqref{eq:ACP} with ${f \coloneqq u \in
		L^p(J; X)}$. Moreover, since
	$v \in W^{1,p}(J; X) \cap L^p(J; X)$, the conditions
	of 
	Proposition~\ref{prop:ACP-strong-vs-mild}\ref{prop:ACP-strong-vs-mild:b}--\ref{prop:ACP-strong-vs-mild:c}
	are satisfied, so
	that $v$ is a strong solution, which proves the base case.
	
	Now let $k \in \N$ and suppose that~\ref{prop:ID-DI-s:DI} holds for $s 
	= k$.
	If $\fI^{k+1} u \in \dom{\fD^{k+1}}$, then by definition we have 
	$\fI^{k+1} u \in \dom{\fD^k}$ and $\fD^k \fI^{k+1} u \in \dom{\fD^1}$. 
	In 
	view of
	Proposition~\ref{prop:frac-par-int-props}\ref{prop:frac-par-int-props:semigroup},
	this means that $\fI^{k} \fI^1 u \in \dom{\fD^k}$
	and $\fD^k \fI^{k} \fI^1 
	u 
	\in \dom{\fD^1}$.
	Combining the former expression with the induction hypothesis yields
	\begin{equation}\label{eq:DkIkIone}
		\fD^k \fI^{k+1} u = \fD^k \fI^{k} \fI^1 u = \fI^1 u, 
		\quad 
		\text{a.e.},
	\end{equation}
	and thus $\fI^1 u \in \dom{\fD^1}$ 
	by the latter. It follows that $\fD^1 \fI^1 u = u$ a.e.\ by the base 
	case 
	of~\ref{prop:ID-DI-s:DI}. Using~\eqref{eq:DkIkIone} once more, this 
	becomes 
	$\fD^1 \fD^k \fI^{k+1} u = u$ a.e., whose left-hand side equals  
	$\fD^{k+1} \fI^{k+1} u$ a.e.\ by the definition of integers powers of 
	$\fD$.
	This proves~\ref{prop:ID-DI-s:DI} for $s = k + 1$, and thus for 
	all $s \in \N$ by induction.
	
	For $s \in (0,\infty) \setminus \N$, the
	assertion follows upon combining the
	definition~\eqref{eq:def-RL-fracpar-deriv-R} of $\mathfrak D^s$ with 
	Proposition~\ref{prop:frac-par-int-props}\ref{prop:frac-par-int-props:semigroup}
	and the integer case:
	\begin{equation}
		\fD^s \fI^s u 
		=
		(\partial_t + A)^{\lceil s \rceil} \fI^{\lceil s \rceil - s} \fI^s u 
		=
		(\partial_t + A)^{\lceil s \rceil} \fI^{(\lceil s \rceil - s) + s} u 
		=
		u \quad \text{a.e.}
	\end{equation}

	\ref{prop:ID-DI-s:ID} The case $s = 1$ follows from
	Proposition~\ref{prop:ACP-strong-vs-mild}
	\ref{prop:ACP-strong-vs-mild:a}$\implies$\ref{prop:ACP-strong-vs-mild:b} 
	with 
	$f \coloneqq u' + Au$, and the integer case follows by 
	induction. For fractional $s$, fix $u \in \dom{\fD^s}$
	and note
	\begin{equation}
		\fI^{\lceil s \rceil - s} \fI^s \fD^s u
		=
		\fI^{\lceil s \rceil - s} \fI^s 
		(\partial_t + A)^{\lceil s \rceil}\fI^{\lceil s \rceil - s} u
		=
		\fI^{\lceil s \rceil}
		(\partial_t + A)^{\lceil s \rceil}\fI^{\lceil s \rceil - s} u
		=
		\fI^{\lceil s \rceil - s} u,
	\end{equation}
	holds 
	a.e.
	Since~\ref{prop:ID-DI-s:DI} implies that $\fI^{\lceil s \rceil - s}$ 
	is injective, we conclude $\fI^s \fD^s u = u$ a.e.
\end{proof}
Combining 
Propositions~\ref{prop:frac-par-int-props}\ref{prop:frac-par-int-props:Lp-Cub} 
and~\ref{prop:ID-DI-s}\ref{prop:ID-DI-s:ID} yields the following corollaries:
\begin{corollary}\label{cor:dom-Ds-continuity}
	Suppose that 
	Assumption~\ref{ass:bdd-semigroup}\ref{ass:bdd-semigroup:exp} 
	holds. If $s \in [0,\infty)$ and $p \in [1,\infty]$ 
	satisfy~\eqref{eq:sp-Ds-cont}, then we have $\dom{\fD^s} \subseteq \Cb(\R; 
	X)$.
\end{corollary}
\begin{corollary}\label{cor:Lp-implies-mild-ACP-R}
	Suppose that 
	Assumption~\ref{ass:bdd-semigroup}\ref{ass:bdd-semigroup:exp} 
	is satisfied and let $u$ be an $L^p$-solution to~\eqref{eq:frac-ACP-R} in 
	the sense of 
	Definition~\ref{def:Lpsol-ACP-R} for some $s \in (0,\infty)$, $p \in 
	[1,\infty]$. If $s$ and $p$ 
	satisfy~\eqref{eq:sp-Ds-cont}, then $u$ is the 
	mild solution in the 
	sense of 
	Definition~\ref{def:mildsol-ACP-R}.
\end{corollary}
\begin{proof}
	If $u$ is an $L^p$-solution, then $u \in \dom{\fD^s}$ 
	and $\fD^s u = f$ holds almost everywhere. Thus, by 
	Proposition~\ref{prop:ID-DI-s}\ref{prop:ID-DI-s:ID}, we can apply $\fI^s$ 
	on both sides to obtain $u = \fI^s f$ a.e., 
	and we have $u \in \Cb(\R; X)$ by Corollary~\ref{cor:dom-Ds-continuity} (or 
	by 
	Proposition~\ref{prop:frac-par-int-props}\ref{prop:frac-par-int-props:Lp-Cub}
	directly).
\end{proof}

\section{Dirichlet problem associated to the fractional parabolic 
	derivative operator}
\label{sec:Dirichlet-frac-parab}

In this section we turn to the main subject of the present work, namely 
the natural
abstract Dirichlet problem associated to $\fD^s$, which consists in 
finding a 
function
$u \from \R \to X$ satisfying
\begin{equation}\label{eq:natural-Dirichlet-Ds}
	\left\lbrace
	\begin{aligned}
		\mathfrak D^s u(t) &= 0, &&t \in (t_0,\infty), \\
		u(t) &= g(t), &&t \in (-\infty, t_0],
	\end{aligned}
	\right.
\end{equation}
for $s \in (0,\infty)$, $t_0 \in \R$ and sufficiently regular
$g \from (-\infty, t_0] \to X$. Recall from 
Section~\ref{section:prelims:frac-par-calc} that $\fD^s$ denotes the 
Riemann--Liouville type fractional parabolic differentiation operator acting on 
functions from $\R$ to $X$, which is our interpretation of the operator 
$(\partial_t + 
A)^s$ appearing in~\eqref{eq:natural-Dirichlet-abstract}, as motivated by 
Remark~\ref{rem:motivation-Ds}.

As in the previous sections, we begin by defining the notion of an 
$L^p$-solution to~\eqref{eq:natural-Dirichlet-Ds} 
(cf.~Definitions~\ref{def:Lpsol-firstorderACP} and~\ref{def:Lpsol-ACP-R}), 
and subsequently define the mild solution 
(cf.~Definitions~\ref{def:ACP-mild-sol} and~\ref{def:mildsol-ACP-R}), which is 
the rigorous formulation of the solution formula formally given 
by~\eqref{eq:mildsol-Dirichlet}. As before, we note that the existence and 
uniqueness of the mild solution are immediate from the definition; for the 
$L^p$-solution, we have uniqueness but the matter of existence is outside of 
the scope of this work, analogously to the discussion below 
Definition~\ref{def:Lpsol-ACP-R}.

\begin{definition}[$L^p$-solution]\label{def:Lpsol-Dirichlet}
	Let
	Assumption~\ref{ass:bdd-semigroup}\ref{ass:bdd-semigroup:exp} 
	be satisfied and suppose that $s \in (0,\infty)$, $p \in [1,\infty]$, $t_0 \in \R$ and 
	$g \in L^p(-\infty, t_0; X)$. Then $u \in L^p(\R; X)$ is called an 
	\emph{$L^p$-solution} 
	to~\eqref{eq:natural-Dirichlet-Ds} if $u \in \dom{\fD^s}$ and both 
	equations in~\eqref{eq:natural-Dirichlet-Ds} hold almost everywhere on 
	their respective sub-intervals of $\bbR$. In particular, we have 
	$g \in \dom{\fD^s}$ on $(-\infty, t_0]$.
\end{definition}

\begin{definition}[Mild solution]\label{def:mildsol-Dirichlet}
	Let Assumption~\ref{ass:bdd-semigroup}\ref{ass:bdd-semigroup:bdd} 
	be satisfied, suppose that $s \in (0,\infty)$, ${p \in [1,\infty]}$, ${t_0 \in 
		\R}$ are given and let $g \in \Cb((-\infty, t_0]; X) \cap 
	\dom{\fD^{(s - 1) \vee 
			0}}$ be such that
	$\fD^{(s - 1) \vee 0}g \in \Cb((-\infty, t_0]; X)$. The \emph{mild 
		solution} to~\eqref{eq:natural-Dirichlet-Ds} with initial datum $g$ 
	is the 
	function $u \in \Cb(\R \setminus \{t_0\}; X)$ defined by
	\begin{equation}
		u(t) \coloneqq g(t), \quad t \in (-\infty, t_0],
	\end{equation}
	and, for $s \in (0,\infty) \setminus \N$,
	\begin{equation} \label{eq:Es}
		\begin{aligned}
			u(t)
			\coloneqq 
			{}&\frac{\sin(\pi\{s\})}{\pi}
			\int_0^\infty 
			\frac{\tau^{-\{s\}}}{\tau + 1} S((t-t_0)(\tau + 1)) g(t_0 - 
			(t-t_0)\tau)
			\rd \tau
			\\&+
			\sum_{k=1}^{\lfloor s \rfloor}
			\frac{(t-t_0)^{\{s\} + k - 1}}{\Gamma(\{s\} + k)}
			S(t-t_0) 
			\fD^{\{s\} + k - 1} g(t_0), 
			\quad 
			t \in (t_0, \infty),
		\end{aligned}
	\end{equation}
	whereas for $s = n \in \N$, we set 
	\begin{equation} \label{eq:Es-n}
		u(t)
		\coloneqq 
		\sum_{k=0}^{n-1}
		\frac{(t-t_0)^{k}}{k!}
		S(t-t_0) 
		\fD^{k} g(t_0), 
		\quad 
		t \in (t_0, \infty).
	\end{equation}
\end{definition}

The following proposition shows that the mild solution is indeed 
well-defined in the sense that it possesses the continuity properties asserted 
in Definition~\ref{def:mildsol-Dirichlet}. 
Its proof is postponed to Section~\ref{sec:Dirichlet-frac-parab:mildsol-props}, 
in which we also state and prove some additional properties of the mild 
solution.

\begin{proposition}\label{prop:mildsol-Dirichlet-bound-cont}
	Suppose Assumption~\ref{ass:bdd-semigroup}\ref{ass:bdd-semigroup:bdd} holds.
	Let $s \in (0,\infty)$, $p \in [1,\infty]$, ${t_0 \in \R}$
	be given and let $g \in L^p(-\infty, t_0; X)$ be as in 
	Definition~\ref{def:mildsol-Dirichlet}. Then the mild solution $u$ 
	to~\eqref{eq:natural-Dirichlet-Ds} satisfies $u \in \Cb(\R \setminus 
	\{t_0\}; X)$ 
	and, for all $t \in (t_0, \infty)$,
	\begin{equation}
		\| u(t) \|_X
		\le
		M_0 \overline\Gamma(s, w(t-t_0)) 
		\max \{\| g \|_{\infty}, \norm{\fD^{\{s\}} g(t_0)}{X}, 
		\dots, \norm{\fD^{s-1} g(t_0)}{X}\},
		\label{eq:bound-Egamma-IncGamma}
	\end{equation}
	where $M_0 \in [1,\infty)$ and $w \in [0,\infty)$ are as 
	in~\eqref{eq:semigroup-est}.
	
	If moreover $g(t), \fD^{\{s\} + k} g(t_0) \in \clos{\dom{A}}$ for all 
	$t \in (-\infty, t_0]$ and $k \in \{0, \dots, 
	\lfloor s \rfloor\}$, 
	then we in fact have $u \in \Cb(\R; X)$.
\end{proposition}
\begin{remark}
	Let us emphasize that the solution formula can fail to be continuous 
	at 
	$t_0$ even in the first-order case $u(t) = S(t-t_0) x$ if $x \not\in 
	\clos{\dom{A}}$.
	As an example, we can take $X = \Cb(\R)$, $A = -\Delta$ and $x(\xi) = 
	\sin(\xi^2)$. Then $-A$ generates the analytic heat semigroup and
	$\clos{\dom{A}} = \Cub(\R)$ is the space of bounded and uniformly 
	continuous 
	functions on $\R$, cf.~\cite[Corollary~3.1.9]{Lunardi1995}. In this case, 
	$\norm{u(t)}{\infty} \le 1$ for all $t \in 
	[t_0, \infty)$, but $S(t - t_0)x$ does not converge uniformly to 
	$x$ as $t \downarrow t_0$.
\end{remark}

The motivation for the solution formulae in 
Definition~\ref{def:mildsol-Dirichlet} is provided by the following theorem, 
which shows that any $L^p$-solution to~\eqref{eq:natural-Dirichlet-Ds} is a 
mild solution whenever $s \in (0,\infty)$ and $p \in [1,\infty]$ are 
such that~\eqref{eq:Es}--\eqref{eq:Es-n} are meaningful. 
\begin{theorem}\label{thm:Lp-implies-mild}
	Suppose that 
	Assumption~\ref{ass:bdd-semigroup}\ref{ass:bdd-semigroup:exp} is 
	satisfied and let $u$ be an $L^p$-solution 
	to~\eqref{eq:natural-Dirichlet-Ds} in the sense of 
	Definition~\ref{def:Lpsol-Dirichlet} for some $p \in [1,\infty]$, 
	$s \in (0, \infty)$ and $t_0 \in \R$.
	If $s$ and $p$ satisfy~\eqref{eq:sp-Ds-cont}, then $u$ is the 
	mild solution to~\eqref{eq:natural-Dirichlet-Ds} in the sense of 
	Definition~\ref{def:mildsol-Dirichlet}.
\end{theorem}
The proof of Theorem~\ref{thm:Lp-implies-mild} is presented in 
Section~\ref{sec:Dirichlet-frac-parab:Lp-implies-mild}, where the integer-order 
and fractional-order cases are treated separately. Before proceeding to the 
next subsection, we consider the following important example of a 
situation in which we can write down an explicit mild solution formula 
for~\eqref{eq:ACP}:
\begin{example}[The fractional heat operator $(\partial_t - \Delta)^s$ on 
	$L^2(\R^d)$]
	Let us consider the function space $X = L^2(\R^d)$ and differential 
	operator $A = -\Delta$ for $d \in \N$, i.e., the 
	negative Laplacian on the full Euclidean space $\R^d$.
	By classical results, see for 
	instance~\cite[Section~13.6.(c)]{Neerven2022}, we know that $-A = \Delta$ 
	generates the \emph{heat semigroup} $(S(t))_{t\ge0}$, which is given by the 
	(spatial) convolution with the \emph{Gauss--Weierstrass kernel}. That is, 
	for all $t \in (0,\infty)$, 
	$f 
	\in L^2(\R^d)$ and $x \in \R^d$, we have
	\begin{equation}
		[S(t)f](x)
		=
		\int_{\R^d} K_t(x - y) f(y) \rd y,
	\end{equation}
	where $K_t \from \R^d \to \R$ is defined by
	\begin{equation}
		K_t(x) \coloneqq \frac{1}{(\sqrt{4\pi t})^d} 
		\exp\biggl(-\frac{\norm{x}{\R^d}^2}{4t}\biggr).
	\end{equation}
	Substituting these formulae into equation~\eqref{eq:Es} for some 
	sufficiently 
	regular 
	initial datum function $g \from (-\infty, t_0] \times \R^d \to \R$, we 
	obtain an explicit 
	formula for the mild solution $u \from \R^{d+1} \to \R$ 
	to~\eqref{eq:natural-Dirichlet-Ds}.
	For example, if $t_0 = 0$ and $s \in (0,1)$, then we have for all 
	$(t, x) \in (0,\infty) \times \R^d$:
	\begin{equation}
		\begin{aligned}
			u(t, x)
			&= 
			\frac{\sin(\pi s)}{\pi}
			\int_0^\infty 
			\int_{\R^d}
			\frac{\tau^{-s}}{\tau + 1} 
			K_{t(\tau + 1)}(x - y) g(-t\tau, y)
			\rd y
			\rd \tau
			\\
			&=
			\frac{\sin(\pi s)}{\pi(\sqrt{4\pi t})^d}
			\int_0^\infty 
			\int_{\R^d}
			\tau^{-s} (\tau+1)^{-\frac{d}{2} - 1}
			\exp\biggl(-\frac{\norm{x-y}{\R^d}^2}{4t(\tau + 1)}\biggr)
			g(-t\tau, y)
			\rd y
			\rd \tau.
		\end{aligned}
	\end{equation}
	If $s \in (1,2)$ (still with $t_0 = 0$), then we instead find for all 
	$t \in (0, \infty)$ and $x \in \R^d$:
	\begin{align}
			u(t, x)
			={} &\frac{\sin(\pi \{s\})}{\pi(\sqrt{4\pi t})^d}
			\int_0^\infty 
			\int_{\R^d}
			\tau^{-\{s\}} (\tau+1)^{-\frac{d}{2} - 1}
			\exp\biggl(-\frac{\norm{x-y}{\R^d}^2}{4t(\tau + 1)}\biggr) 
			g(-t\tau, y)
			\rd y
			\rd \tau\\[-1em]
			&+
			\frac{t^{\{s\}}}{\Gamma(\{s\} + 1)}
			S(t) 
			\fD^{\{s\}} g(0). 
		\label{eq:mildsol-heat-s12}
	\end{align}
The fractional parabolic derivative 
$\fD^\alpha f$ of any sufficiently regular $f \from \R \to X$
(e.g., $f \in \dom{\fD^1}$) admits the following Marchaud type 
representation for 
$\alpha \in (0,1)$ and $t \in \R$
(cf.~\cite[Proposition~3.2.1]{MartinezCSanzA2001} or 
\cite[Equation~(1.2)]{Stinga2017}):
\begin{equation}
	\fD^{\alpha} f(t) 
	= 
	\frac{1}{\Gamma(-\alpha)}\int_0^\infty 
	\sigma^{-\alpha - 1} [S(\sigma)f(t-\sigma) - f(t)] \rd \sigma.
\end{equation}
Therefore, supposing that $g$ is sufficiently regular, we have
for all $t \in (0,\infty)$,
\begin{align}
	&\frac{t^{\{s\}}}{\Gamma(\{s\} + 1)}
	S(t) 
	\fD^{\{s\}} g(0)
	\\
	&\qquad =
	-\frac{t^{\{s\}} \sin(\pi\{s\})}{\pi}
	\int_0^\infty 
	\sigma^{-\{s\}-1} [S(t + \sigma)g(-\sigma) - S(t)g(0)] \rd \sigma
	\\
	&\qquad =
	-\frac{\sin(\pi\{s\})}{\pi}
	\int_0^\infty 
	\tau^{-\{s\}-1} [S(t(\tau + 1))g(-\tau t) - S(t)g(0)] \rd \tau,
\end{align}
where we used the reflection formula for the gamma 
function~\cite[Equation~(5.5.3)]{Olver2010}
to obtain the prefactor on the second line, and the change of variables 
$\sigma = t \tau$ for the third line. Substituting the heat semigroup 
once more, we obtain, for all $x \in \R^d$,
\begin{equation}
	\begin{aligned}
		&\frac{t^{\{s\}}}{\Gamma(\{s\} + 1)}
		[S(t) \fD^{\{s\}} g(0)](x)
		\\
		&\quad\begin{aligned}[t]
			&=
			\frac{\sin(\pi\{s\})}{\pi}
			\int_0^\infty 
			\int_{\R^d}
			\tau^{-\{s\}-1} [K_t(x-y) g(0, y) - K_{t(1 + \tau)}(x-y)g(-\tau t, 
			y)] 
			\rd y\rd \tau
			\\
			&=
			\frac{\sin(\pi\{s\})}{\pi (\sqrt{4\pi t})^d}
			\int_0^\infty 
			\int_{\R^d} \tau^{-\{s\}-1} 
			\begin{aligned}[t]
				\biggl[
				&\exp\biggl(-\frac{\norm{x-y}{\R^d}^2}{4t}\biggr) 
				g(0, y) 
				\\&- 
				(\tau + 1)^{-\frac{d}{2}} 
				\exp\biggl(-\frac{\norm{x-y}{\R^d}^2}{4t(\tau+1)}\biggr) 
				g(-\tau t, y)
				\biggr] \rd y\rd \tau.
			\end{aligned}
		\end{aligned}
	\end{aligned}
\end{equation}
Substituting this into~\eqref{eq:mildsol-heat-s12}, we 
conclude 
that the mild solution to the Dirichlet problem associated to the 
fractional heat operator $(\partial_t - \Delta)^s$, with $s \in (1,2)$ and 
sufficiently regular initial datum $g$, admits the explicit expression
\begin{equation}
	\begin{aligned}
		u(&t,x) =
		\frac{\sin(\pi \{s\})}{\pi(\sqrt{4\pi t})^d}
		\int_0^\infty \int_{\R^d} 
		\biggl[
		\tau^{-\{s\}-1} 
		\exp\biggl(-\frac{\norm{x-y}{\R^d}^2}{4t}\biggr) 
		g(0, y) 
		\\
		&+
		\bigl(\tau^{-\{s\}} (\tau+1)^{-\frac{d}{2} - 1} 
		-
		\tau^{-\{s\}-1} (\tau + 1)^{-\frac{d}{2}}\bigr)
		\exp\biggl(-\frac{\norm{x-y}{\R^d}^2}{4t(\tau + 1)}\biggr) 
		g(-t\tau, y)
		\biggr]
		\rd y \rd \tau 
	\end{aligned}
\end{equation}
for all $(t, x) \in (0,\infty) \times \R^d$.
\end{example}

\subsection{Properties of the mild solution}
\label{sec:Dirichlet-frac-parab:mildsol-props}
%
In this section, we further investigate the mild solution concept 
introduced in Definition~\ref{def:mildsol-Dirichlet} by
establishing some of its key properties. To this end, we start with the central 
observation that 
formula~\eqref{eq:Es} has close connections to the
normalized upper incomplete gamma function, whose principal branch
$\overline\Gamma(\alpha, \,\cdot\,) \from \C \setminus (-\infty, 0) \to \C$
for $\alpha \in \C_+$ is
defined by
\begin{equation}\label{eq:norm-upp-inc-gamma}
\overline\Gamma(\alpha, z) \coloneqq 
\frac{1}{\Gamma(\alpha)} \int_z^\infty \zeta^{\alpha - 1} e^{-\zeta} \rd 
\zeta, 
\quad z \in \C \setminus (-\infty, 0),
\end{equation}
integrating over any contour from $z$ to $\infty$ avoiding $(-\infty, 0)$, 
see~\cite[Chapter~8]{Olver2010}.

The relation to~\eqref{eq:Es} follows from the following identities.
For $\alpha \in (0,1)$ and $z \in \C_+ \cup \{0\}$, 
\cite[Equations~(5.5.3),~(13.4.4) and~(13.6.6)]{Olver2010} yield
\begin{equation}\label{eq:confHGF-incgamma}
\overline\Gamma(\alpha, z)
=
\frac{\sin(\pi\alpha)}{\pi}
\int_0^\infty 
\frac{\tau^{-\alpha}}{1+\tau} e^{-z(1+\tau)} \rd \tau.
\end{equation}
In particular, for $t \in (0,\infty)$, the change of variables 
$\sigma \coloneqq t(1+\tau)$ produces
\begin{equation}\label{eq:Laplace}
\overline\Gamma(\alpha, tz)
=
\frac{t^\alpha \sin(\pi\alpha)}{\pi} 
\int_t^\infty 
(\sigma - t)^{-\alpha} \sigma^{-1} e^{-\sigma z} \rd \sigma.
\end{equation}
Moreover, for all $\alpha \in (0,\infty)$ and $n \in \N$, we 
have the following recurrence relations~\cite[Equations~(8.8.12) and 
(8.4.10)]{Olver2010}:
\begin{equation}\label{eq:incgamma-recurrence}
\overline\Gamma(\alpha, z)
=
\overline{\Gamma}(\{\alpha\}, z)
+
\sum_{k=1}^{\lfloor \alpha \rfloor} 
\frac{z^{k+\{\alpha\}-1}}{\Gamma(k+\{\alpha\})}e^{-z}, 
\quad 
\overline\Gamma(n, z) 
=
\sum_{k=0}^{n-1} \frac{z^k}{k!} e^{-z}.
\end{equation}
As a first application of these identities, we present the following 
proof:
\begin{proof}[Proof of 
Proposition~\ref{prop:mildsol-Dirichlet-bound-cont}]
The estimate on $\norm{u(t)}{X}$ follows by applying the triangle 
inequality,~\eqref{eq:semigroup-est} and 
identities~\eqref{eq:confHGF-incgamma}--\eqref{eq:incgamma-recurrence} 
to~\eqref{eq:Es}. The continuity assertions rely on the strong continuity 
of $(S(t))_{t\ge0}$, see the remarks below 
Assumption~\ref{ass:bdd-semigroup}. They are immediate for the terms 
involving $\fD^{\{s\} + 
	k} g(t_0)$. For the integral term, we note 
that the norm of the integrand is dominated by
$\tau \mapsto M_0 \norm{g}{\infty} \frac{\tau^{-\{s\}}}{\tau + 1}$, which 
is integrable 
in view of~\eqref{eq:confHGF-incgamma}. Combined with the continuity of 
$t \mapsto S((t-t_0)(\tau + 1)) g(t_0 - (t-t_0))$ on $(t_0, \infty)$ and 
possibly at $t_0$, the dominated convergence theorem yields the result.
\end{proof}
Next we comment on the precise way in which the fractional-order mild 
solution formula~\eqref{eq:Es} of Definition~\ref{def:mildsol-Dirichlet} 
reduces to formula \eqref{eq:Es-n} for integer orders $s = n \in \N$. 
Substituting $s = n$ (i.e., $\lfloor s \rfloor = n$ and $\{s\} = 0$) 
in the higher-order terms of~\eqref{eq:Es} and shifting the index of 
summation yields 
\begin{equation}\label{eq:desired-integer}
\sum_{k=0}^{n-1}
\frac{(t-t_0)^{k}}{k!}
S(t-t_0) 
[(\partial_t + A)^{k} g](t_0),
\quad 
\forall t \in (t_0, \infty),
\end{equation}
as desired. 
Moreover, the first term in~\eqref{eq:Es} vanishes as required, provided that 
the integral remains convergent for $\{s\} = 0$. This occurs under 
Assumption~\ref{ass:bdd-semigroup}\ref{ass:bdd-semigroup:exp}, but may fail in 
general if only Assumption~\ref{ass:bdd-semigroup}\ref{ass:bdd-semigroup:bdd} 
is satisfied, hence in this case we cannot argue via direct substitution. 
Instead, we have to consider limits as $s \to n$. Let $u_s$ denote the 
mild solution from Definition~\ref{def:mildsol-Dirichlet} of order $s \in 
(0,\infty) \setminus \N$. Then
\begin{align}
u_{n + \varepsilon}(t)
= 
{}&\frac{\sin(\pi \varepsilon)}{\pi}
\int_0^\infty 
\frac{\tau^{-\varepsilon}}{\tau + 1} S((t-t_0)(\tau + 1)) g(t_0 - 
(t-t_0)\tau)
\rd \tau \label{eq:unpluseps-int}
\\&+
\sum_{k=0}^{n-1}
\frac{(t-t_0)^{k + \varepsilon}}{\Gamma(k + \varepsilon + 1)}
S(t-t_0) 
\fD^{k + \varepsilon} g(t_0), \label{eq:unpluseps-sum}
\\
u_{n - \varepsilon}(t)
= 
{}&\frac{\sin(\pi \varepsilon)}{\pi}
\int_0^\infty 
\frac{\tau^{\varepsilon - 1}}{\tau + 1} S((t-t_0)(\tau + 1)) g(t_0 - 
(t-t_0)\tau)
\rd \tau \label{eq:unminuseps-int}
\\&+
\sum_{k=1}^{n - 1}
\frac{(t-t_0)^{k - \varepsilon}}{\Gamma(k - \varepsilon + 1)}
S(t-t_0) 
\fD^{k - \varepsilon} g(t_0), \label{eq:unminuseps-sum}
\end{align}
for all $\varepsilon \in (0,1)$ and $t \in (t_0, \infty)$.
Substituting $\varepsilon = 0$ into the summations on 
lines~\eqref{eq:unpluseps-sum} and~\eqref{eq:unminuseps-sum}, and comparing the 
resulting expressions 
with~\eqref{eq:desired-integer}, we see that the integral terms on 
lines~\eqref{eq:unpluseps-int} and~\eqref{eq:unminuseps-int} should 
converge to zero and $S(t - t_0)g(t_0)$, respectively, as $\varepsilon \to 0$, 
in 
order to recover the integer-order case (formally, since we 
cannot expect the continuity of $\varepsilon \mapsto \mathfrak D^{k \pm 
\varepsilon} g(t_0)$ in general).
The following proposition states when these convergences hold:

\begin{proposition}\label{prop:reduction}
Let $t_0 \in \R$ and $g \in \Cb((-\infty, t_0]; X)$ be given. If 
Assumption~\ref{ass:bdd-semigroup}\ref{ass:bdd-semigroup:bdd} is satisfied, 
then for all 
$t \in (t_0, \infty)$ it holds that 
\begin{equation}
\frac{\sin(\pi \varepsilon)}{\pi}
\int_0^\infty 
\frac{\tau^{\varepsilon - 1}}{\tau + 1} S((t-t_0)(\tau + 1)) g(t_0 - 
(t-t_0)\tau)
\rd \tau
\to S(t-t_0)g(t_0)
\quad 
\text{as } \varepsilon \to 0.
\end{equation}
If, in addition,
Assumption~\ref{ass:bdd-semigroup}\ref{ass:bdd-semigroup:exp} is satisfied, 
then
\begin{equation}\label{eq:Prop33-Assertion2}
\frac{\sin(\pi \varepsilon)}{\pi}
\int_0^\infty 
\frac{\tau^{-\varepsilon}}{\tau + 1} S((t-t_0)(\tau + 1)) g(t_0 - 
(t-t_0)\tau)
\rd \tau
\to 0
\quad 
\text{ as } \varepsilon \to 0.
\end{equation}
\end{proposition}
\begin{proof}
Fix $t_0 \in \R$ and $t \in (t_0, \infty)$.
First we define the function $f_{t, t_0} \from \R\to X$ by 
\begin{equation}
f_{t, t_0}(r) \coloneqq 
\begin{cases}
	S(t - t_0 - r) g(t_0 + r), &r \in (-\infty, 0]; \\
	S(t - t_0) g(t_0), &r \in (0,\infty),
\end{cases}
\end{equation}
which is bounded and continuous at $r = 0$ by the assumptions on 
$(S(t))_{t\ge0}$ and~$g$.
Next, for any $\varepsilon \in (0,1)$ we define $\psi_{t, t_0, \varepsilon} 
\from 
\R\to 
[0,\infty)$ by 
\begin{equation}
\psi_{t, t_0, \varepsilon}(r) 
\coloneqq 
\frac{(t - t_0)^{1 - \varepsilon} \sin(\pi \varepsilon)}{\pi}
(r + (t - t_0))^{-1}_+ r^{\varepsilon - 1}_+, 
\quad 
r \in \R.
\end{equation}
Shifting the integration variable by $t - t_0$ and 
applying~\eqref{eq:Laplace} with $s = 1 - \varepsilon$ and $z = 0$, we find
\begin{align}
\int_\R \psi_{t, t_0, \varepsilon}(r) \rd r
=
\frac{(t - t_0)^{1 - \varepsilon} \sin(\pi \varepsilon)}{\pi}
\int_{t - t_0}^\infty 
r^{-1} (r - (t-t_0))^{\varepsilon - 1} \rd r
= 1.
\end{align}
Moreover, we have for any $\delta >0$:
\begin{equation}
\begin{aligned}
	&\int_{\{ \lvert r \rvert \ge \delta \}}
	\lvert 
	\psi_{t, t_0, \varepsilon}(r)
	\rvert
	\rd r
	=
	\frac{(t - t_0)^{1-\varepsilon}\sin(\pi \varepsilon)}{\pi} 
	\int_\delta^\infty 
	(r + t - t_0)^{-1} r^{\varepsilon - 1} \rd r
	\\
	&\quad\le 
	\frac{(t - t_0)^{1-\varepsilon}\sin(\pi \varepsilon)}{\pi} 
	\int_\delta^\infty 
	r^{\varepsilon - 2} \rd r
	=
	\frac{(t - t_0)^{1-\varepsilon}\delta^{\varepsilon - 1}\sin(\pi 
		\varepsilon)}{\pi(1 - \varepsilon)}
	\to
	\frac{(t - t_0)\cdot 0}{\pi \delta} 
	= 0
\end{aligned}
\end{equation}
as $\varepsilon \to 0$. Together, these observations show that the family 
$(\psi_{t, t_0, \varepsilon})_{\varepsilon\in(0,1)}$ forms an approximate 
identity as $\varepsilon \to 0$ in the sense 
of~\cite[Definition~1.2.15]{Grafakos2008}.
Since the change of variables $\sigma \coloneqq (t - 
t_0)\tau$ yields 
\begin{equation}
\frac{\sin(\pi \varepsilon)}{\pi}
\int_0^\infty 
\frac{\tau^{\varepsilon - 1}}{\tau + 1} S((t-t_0)(\tau + 1)) g(t_0 - 
(t-t_0)\tau)
\rd \tau
=
[\psi_{t, t_0, \varepsilon} * f_{t, t_0}](0),
\end{equation}
the first assertion now follows by applying the obvious vector-valued 
generalization 
of~\cite[Theorem~1.2.19(2)]{Grafakos2008}, which gives
\[
[\psi_{t, t_0, \varepsilon} * f_{t, t_0}](0) \to f_{t, t_0}(0) = 
S(t-t_0)g(t_0)
\quad 
\text{ as } \varepsilon \to 0.
\]

For the second assertion, suppose that 
Assumption~\ref{ass:bdd-semigroup}\ref{ass:bdd-semigroup:exp} holds. 
By Proposition~\ref{prop:mildsol-Dirichlet-bound-cont}, the left-hand side 
of~\eqref{eq:Prop33-Assertion2} is bounded above by $M_0 
\overline\Gamma(\varepsilon, 
w(t - t_0)) 
\norm{g}{\infty}$ 
for all $t \in (t_0, \infty)$. Since $w(t - t_0) > 0$, this expression 
tends to zero as $\varepsilon \to 0$.
\end{proof}
The final corollary concerns the choice $g \equiv x 
\in 
\dom{A^{(s-1) \vee 0}}$ in Definition~\ref{def:mildsol-Dirichlet}, in which 
case the solution can be expressed in terms of an operator-valued 
counterpart of the upper 
incomplete gamma function.
Namely, for $\alpha \in (0,1)$ and $t \in (0,\infty)$, we use the Phillips 
calculus 
from Section~\ref{section:prelims:flips} to define
\begin{equation}\label{eq:def-incgammaop-01}
\overline\Gamma(\alpha, tA) 
\coloneqq [z \mapsto \overline\Gamma(\alpha, tz)](A)
= 
\cL\biggl[ 
\frac{t^\alpha \sigma^{-1}_+ (\sigma - 
t)^{-\alpha}_+}{\Gamma(\alpha)\Gamma(1-\alpha)} 
\rd\sigma
\biggr](A)
\in \LO(X),
\end{equation}
see~\eqref{eq:Laplace}. For 
$\alpha \in [1, \infty)$, such a Laplace 
transform representation 
is no longer available, so in this case we instead define 
$\overline\Gamma(\alpha, tA)$
by analogy with~\eqref{eq:incgamma-recurrence}:
\begin{equation}\label{eq:def-incgamma-s1infty}
\overline\Gamma(\alpha, tA)x
\coloneqq
\overline{\Gamma}(\{\alpha\}, tA)x
+
\sum_{k=1}^{\lfloor \alpha \rfloor} 
\frac{t^{k+\{\alpha\}-1}}{\Gamma(k+\{\alpha\})} A^{k+\{\alpha\}-1} S(t)x, 
\quad
x \in \dom{A^{\alpha - 1}}.
\end{equation}
For $t = 0$ we set $\overline\Gamma(\alpha, 0 A) \coloneqq \id_X$.
Although $\overline\Gamma(\alpha, tA)$ is unbounded in 
general, under the additional
Assumption~\ref{ass:bdd-semigroup}\ref{ass:bdd-semigroup:ana}
we have, cf.~\cite[Chapter~2, Theorem~6.13(c)]{Pazy1983}:
\begin{equation}\label{eq:analytic-semigroup-est}
\forall \beta \in [0, \infty), \:
\exists M_\beta \in [1,\infty) : \:\:
\norm{A^\beta S(t)}{\LO(X)} \le M_\beta t^{-\beta} e^{-wt}, 
\quad \forall t \in (0,\infty).
\end{equation}
Putting these observations together, we obtain the following formula for the 
solution with initial datum $g \equiv x$:
\begin{corollary}\label{cor:g-equiv-x}
Suppose Assumption~\ref{ass:bdd-semigroup}\ref{ass:bdd-semigroup:bdd} holds.
Let $s \in (0,\infty) \setminus \N$, $p \in [1,\infty]$ and ${t_0 \in \R}$.
If $g \equiv x$ for some given $x \in \dom{A^{(s-1) \vee 0}}$ and $s \in 
(0,1)$ or Assumption~\ref{ass:bdd-semigroup}\ref{ass:bdd-semigroup:exp} 
holds, then the 
solution $u$ 
to~\eqref{eq:natural-Dirichlet-Ds} from
Definition~\ref{def:mildsol-Dirichlet} becomes
\begin{equation}\label{eq:mildsol-gx}
u(t) 
= 
\overline\Gamma(s, (t - t_0)A)x, 
\quad 
\forall t \in (t_0, \infty).
\end{equation}
If, in addition, $s \in (0,1)$ or 
Assumption~\ref{ass:bdd-semigroup}\ref{ass:bdd-semigroup:ana} is satisfied, 
then
\begin{equation}\label{eq:mildsol-gx-est}
\norm{u(t)}{X} 
\le 
\biggl[
M_0
\overline\Gamma(\{s\}, w(t - t_0))
+ 
e^{-w(t - t_0)}
\sum_{k=1}^{\lfloor s \rfloor}
\frac{M_{\{s\} + k - 1}}{\Gamma(\{s\} + k)}\biggr] 
\norm{x}{X}, 
\end{equation}
where $w \in [0,\infty)$ and $M_0, M_{\{s\}}, \dots, M_s \in [1,\infty)$ 
are as in~\eqref{eq:semigroup-est} 
and~\eqref{eq:analytic-semigroup-est}.
\end{corollary}
\begin{proof}
The substitution $g \equiv x$ and 
the change of variables $\sigma \coloneqq (t - t_0)(1 + \tau)$ in the first 
term of~\eqref{eq:Es} produces 
\begin{equation}
\frac{\sin(\pi\{s\})}{\pi} (t-t_0)^{\{s\}}
\int_{t-t_0}^\infty \sigma^{-1} (\sigma - (t-t_0))^{-\{s\}}
S(\sigma) \rd \sigma \,
x,
\end{equation}
which is equal to $\overline\Gamma(\{s\}, (t - t_0)A)x$ 
by~\eqref{eq:Laplace}. 

Now suppose that $s \in (1,\infty)$ and let 
Assumption~\ref{ass:bdd-semigroup}\ref{ass:bdd-semigroup:exp}
be satisfied. Then for any $\alpha \in [0,\infty)$ we have
\[
\fI^\alpha g(t) = \frac{1}{\Gamma(\alpha)} \int_0^\infty 
\tau^{\alpha - 1} S(\tau) x \rd \tau = A^{-\alpha}x
\]
by~\eqref{eq:neg-frac-power}, so that
\[
\fD^\beta g(t) 
= 
(\partial_t + A)^{\lceil \beta \rceil} 
\fI^{\lceil \beta \rceil - \beta} g(t)
=
A^{\lceil \beta \rceil} 
A^{\beta - \lceil \beta \rceil} x
= A^{\beta}x
\]
for all $\beta \in [0, s - 1]$, hence the remaining terms of~\eqref{eq:Es} 
become 
\begin{equation}
\sum_{k=1}^{\lfloor s \rfloor}
\frac{(t-t_0)^{\{s\} + k - 1}}{\Gamma(\{s\} + k)}
S(t-t_0) 
A^{\{s\} + k - 1} x,
\end{equation}
proving~\eqref{eq:mildsol-gx} in view of~\eqref{eq:def-incgamma-s1infty}.
Estimate~\eqref{eq:mildsol-gx-est} follows from~\eqref{eq:semigroup-est} 
(and~\eqref{eq:analytic-semigroup-est}).
\end{proof}
%
\subsection{Proof of the relation between mild solutions and \texorpdfstring{$L^p$}{L\^{}p}-solutions}
\label{sec:Dirichlet-frac-parab:Lp-implies-mild}

The aim of this section is to prove Theorem~\ref{thm:Lp-implies-mild}. We 
will prove the 
integer-order and fractional-order cases separately in the following two 
subsections. Before proceeding to do so, we make a preliminary observation 
which applies to both cases:

If Assumption~\ref{ass:bdd-semigroup}\ref{ass:bdd-semigroup:exp} holds, and
$u \in \dom{\fD^s}$ is an $L^p$-solution 
to~\eqref{eq:natural-Dirichlet-Ds} for some $t_0 \in \R$, $p \in [1,\infty]$ and
$s \in (0, \infty)$ satisfying~\eqref{eq:sp-Ds-cont}, then 
$u \in \Cb(\R; X)$ by Corollary~\ref{cor:dom-Ds-continuity}. Moreover, if 
$s \ge 1$, then we have $\fD^{s-1} u \in \dom{\fD^1}$ by the 
definition~\eqref{eq:def-RL-fracpar-deriv-R}, and thus $\fD^{s-1} u \in 
\Cb(\R; X)$ by applying Corollary~\ref{cor:dom-Ds-continuity} once more. Since 
$g \equiv u$ on $(-\infty, t_0]$, this shows that the continuity properties of 
$g$ are as in Definition~\ref{def:mildsol-Dirichlet}.

\subsubsection{Integer-order case}
\label{sec:Lp-implies-mild:integer-order}

Let us consider the case $s = n \in \N$, in which the operator $\fD^n = 
(\partial_t + A)^n$ is local in time. Note that the first-order case $s = 1$ 
was already treated in Section~\ref{sec:11st-order-ACP}. The following 
proposition is the key ingredient in the proof of 
Theorem~\ref{thm:Lp-implies-mild} for $s = n$.

\begin{proposition}\label{prop:InDn}
Suppose that Assumption~\ref{ass:bdd-semigroup}\ref{ass:bdd-semigroup:exp} 
holds. Let $n \in \N$, $p \in [1,\infty]$, $t_0 \in \R$ and 
$u \in \dom{(\partial_t + A)^n}$. 
For all $t \in (t_0, \infty)$, we have
\begin{equation}\label{eq:InDn}
\mathfrak I^n_{t_0} (\partial_t + A)^n u(t)
= 
u(t) 
-
\sum_{k=0}^{n-1} 
\frac{(t-t_0)^k}{k!} S(t-t_0) [(\partial_t + A)^k u](t_0).
\end{equation}
Here, $\fI_{t_0}^s \in \LO(L^p(t_0, \infty; X))$ denotes the 
(Riemann--Liouville type) fractional 
parabolic 
integral, defined by
\begin{equation}\label{eq:def-RL-int}
\fI_{t_0}^s u(t) 
\coloneqq 
\frac{1}{\Gamma(s)}\int_{t_0}^t (t-\tau)^{s-1} S(t-\tau) u(\tau) \rd 
\tau, 
\quad 
u \in L^p(J; X), \; 
\text{a.e.\ } t \in (t_0, \infty).
\end{equation}
\end{proposition}
\begin{proof}
We use induction on $n \in \N$.
For the base case $n = 1$, let us fix an arbitrary
$u \in W^{1,p}(J; X) \cap L^p(J; \dom{A}) \hookrightarrow 
\Cb(\clos{J}; X)$. 
Since $u(t) \in \dom{A}$ a.e., we find in particular that $u(t_0) \in 
\clos{\dom{A}}$.
Now the result follows by applying 
Proposition~\ref{prop:ACP-strong-vs-mild}\ref{prop:ACP-strong-vs-mild:a}$\implies$\ref{prop:ACP-strong-vs-mild:b}
with 
$f \coloneqq u' + Au \in L^p(J; X)$.

Now suppose that the statement is true for some $n \in \N$.
We present the argument for $t_0 = 0$, the other cases being analogous.
Fix $u \in \dom{(\partial_t + A)^{n+1}}$ and 
apply the induction hypothesis to the function
$(\partial_t + A) u \in \dom{(\partial_t + A)^{n}}$,
yielding
\begin{equation}
(\partial_t + A) u(t) 
= 
\mathfrak I^n_{t_0} (\partial_t + A)^{n+1} u(t)
+
\sum_{k=0}^{n-1} \frac{t^k}{k!} S(t) [(\partial_t + A)^{k+1} u](0)
\end{equation}
for all $t \in \clos J$. Applying $\mathfrak I^1$ to both sides of the
above equation, we can use the case $n = 1$ along with 
Propositions~\ref{prop:frac-par-int-props}\ref{prop:frac-par-int-props:semigroup}--%
\ref{prop:frac-par-int-props:invariant-funcs}
to find
\begin{align}
u(t) 
&= 
\mathfrak I^{n+1}_{t_0} (\partial_t + A)^{n+1} u(t)
+
\sum_{k=0}^{n-1} 
\frac{t^{k+1}}{(k+1)!} S(t) [(\partial_t + A)^{k+1} u](0) 
+ 
S(t)u(0)
\\
&=
\mathfrak I^{n+1}_{t_0} (\partial_t + A)^{n+1} u(t)
+
\sum_{k=0}^{n} \frac{t^{k}}{k!} S(t) [(\partial_t + A)^k u](0). 
\tag*{\qedhere}
\end{align}
\end{proof}
We can now prove the integer-order case of 
Theorem~\ref{thm:Lp-implies-mild}:
\begin{proof}[Proof of Theorem~\ref{thm:Lp-implies-mild} (for $s = n \in\N$)]
Let $u$ be an $L^p$-solution to~\eqref{eq:natural-Dirichlet-Ds}. By 
Definition~\ref{def:mildsol-Dirichlet} and the 
observations in the beginning of 
Section~\ref{sec:Dirichlet-frac-parab:Lp-implies-mild}, we have $u \in 
\Cb(\R; X)$, $u \equiv g$ on $(-\infty, t_0]$ and $(\partial_t + A)^n u = 
0$ a.e.\ on $(t_0, \infty)$.
Applying $\fI_{t_0}^s$ to both sides of the latter and using 
Proposition~\ref{prop:InDn} on the left-hand side, we find for all $t \in 
(t_0, \infty)$:
\begin{equation}
u(t) 
=
\sum_{k=0}^{n-1} 
\frac{(t-t_0)^k}{k!} S(t-t_0) [(\partial_t + A)^k u](t_0).
\end{equation}
Note that the operators $(\partial_t + A)^k$ are local in time and 
that, in 
fact, we can choose to interpret $\partial_t$ as a left derivative. Thus, 
since $u \equiv g$ on $(-\infty, t_0]$, we obtain~\eqref{eq:Es-n}.
\end{proof}
If $u \in \dom{(\partial_t + A)^k}$ is sufficiently regular, say if 
$u \in C^j(\clos J; \mathsf D(A^{k-j}))$ ($j$ times continuously differentiable)
for all $j \in \{0,\dots, k\}$, then we have the pointwise binomial 
expansion 
\[
[(\partial_t + A)^k u](t)
=
\sum_{j=0}^{k}
\binom{k}{j} A^{k-j} u^{(j)}(t),
\quad 
\forall 
t \in \clos J,
\]
where $u^{(j)}$ denotes the $j$th (classical) derivative of $u$.
Substituting this into~\eqref{eq:InDn}, using the definition of binomial 
coefficients, interchanging the order of summation and shifting the inner 
summation index yields
\begin{align*}
\sum_{k=0}^{n-1} &\frac{t^k}{k!}
S(t)[(\partial_t + A)^k u](0)
=
\sum_{k=0}^{n-1}
\sum_{j=0}^{k}
\frac{t^k}{j!(k-j)!}
S(t) A^{k-j} u^{(j)}(0)
\\
&=
\sum_{j=0}^{n-1}
\sum_{k=j}^{n-1}
\frac{t^k}{j!(k-j)!}
S(t) A^{k-j} u^{(j)}(0)
=
\sum_{j=0}^{n-1}
\frac{t^j}{j!}
\sum_{\ell=0}^{n-j-1}
\frac{t^\ell}{\ell!}
S(t) A^{\ell} u^{(j)}(0).
\end{align*}
Moreover, note that for $n \in \N$, $x 
\in \dom{A^{n-1}}$ and $t \in 
(0,\infty)$ we have
\begin{equation}
\overline\Gamma(n, tA)x 
= 
\sum_{k=0}^{n-1} \frac{t^k}{k!} S(t) A^k x,
\end{equation}
cf.~\eqref{eq:incgamma-recurrence}. Together, these observations imply 
that, in this situation, equation~\eqref{eq:InDn} takes on the following form:
\begin{equation}\label{eq:Int0-dtAn-Gamma}
\mathfrak I^n_{t_0} (\partial_t + A)^n u(t)
= 
u(t) 
-
\sum_{k=0}^{n-1}
\frac{(t-t_0)^k}{k!}
\overline\Gamma(n-k, (t-t_0)A) x_k.
\end{equation}

\subsubsection{Fractional-order case}

Now we turn to the proof of Theorem~\ref{thm:Lp-implies-mild} in the case 
that $s \in (0,\infty) \setminus \N$. It relies on the following result, which 
is the fractional-order analog to Proposition~\ref{prop:InDn}.

\begin{theorem}\label{thm:extension}
Suppose that Assumption~\ref{ass:bdd-semigroup}\ref{ass:bdd-semigroup:exp} 
holds.
Let $f \in L^p(-\infty, t_0; X)$ for some $p \in [1,\infty]$ and 
$t_0 \in \R$, and let $\widetilde f \in L^p(\R; X)$ denote its 
extension by zero to the whole of $\R$.
Let $s \in (0,\infty) \setminus \N$ be such that~\eqref{eq:sp-Ds-cont} 
is 
satisfied.
Then the following 
identity holds for 
all $t \in (t_0, \infty)$:
\begin{align}
\fI^s \widetilde{f}(t)
= 
\frac{\sin(\pi\{s\})}{\pi}
&\int_0^\infty 
\frac{\tau^{-\{s\}}}{\tau + 1} S((t-t_0)(\tau + 1)) 
\fI^s f(t_0 - (t-t_0)\tau)
\rd \tau
\\&+
\sum_{k=1}^{\lfloor s \rfloor}
\frac{(t-t_0)^{\{s\} + k - 1}}{\Gamma(\{s\} + k)}
S(t-t_0) \fI^{\lfloor s \rfloor - k + 1} f(t_0).
\label{eq:Is-ftilde}
\end{align}
\end{theorem}
Before proving this result, we show how it can indeed be 
used to finish the proof of Theorem~\ref{thm:Lp-implies-mild}:
\begin{proof}[Proof of Theorem~\ref{thm:Lp-implies-mild} ($s \in 
(0,\infty) \setminus \N$)]
Let $u \in \dom{\fD^s}$ be an $L^p$-solution 
to~\eqref{eq:natural-Dirichlet-Ds}. Applying $\fD^s$ to
the second line of~\eqref{eq:natural-Dirichlet-Ds} yields
$
\fD^s u = \fD^s g
$
a.e.\ on $(-\infty, t_0)$.
Combined with the first line, i.e., 
$\fD^s u = 0$ a.e.\ on $(t_0, \infty)$, we 
find
\[
\fD^s u = \widetilde f \quad \text{a.e.\ on } \R,
\]
where $f \coloneqq \fD^s g \in 
L^p(-\infty, t_0; X)$. Now we apply $\fI^s$ to both sides of this 
equation, and use Proposition~\ref{prop:ID-DI-s}\ref{prop:ID-DI-s:ID} and 
Theorem~\ref{thm:extension} to the left-hand and right-hand sides, 
respectively. Together, this yields, for all $t \in (t_0, \infty)$,
\begin{equation}
u(t)
=
\begin{aligned}[t]
	\frac{\sin(\pi\{s\})}{\pi}
	&\int_0^\infty 
	\frac{\tau^{-\{s\}}}{\tau + 1} S((t-t_0)(\tau + 1)) 
	\fI^s [\fD^s g](t_0 - (t-t_0)\tau)
	\rd \tau
	\\&+
	\sum_{k=1}^{\lfloor s \rfloor}
	\frac{(t-t_0)^{\{s\} + k - 1}}{\Gamma(\{s\} + k)}
	S(t-t_0) \fI^{\lfloor s \rfloor - k + 1} [\fD^s g](t_0).
\end{aligned}
\end{equation}
In order to conclude that $u$ satisfies equation~\eqref{eq:Es}, it 
remains to observe that $\fI^s \fD^s g = g$ and 
$\fI^{\lfloor s \rfloor - k + 1} \fD^s g 
= \fD^{s - (\lfloor s \rfloor - k + 1)} g 
= \fD^{\{s\} + k - 1} g$ for all $k \in \{1, \dots, \lfloor s \rfloor\}$. 
Indeed, the former follows from the natural analog of 
Proposition~\ref{prop:ID-DI-s}\ref{prop:ID-DI-s:ID} for functions defined 
on $(-\infty, t_0]$; for the latter, let $m \in \N_0$ be such that 
$m \le \lfloor s \rfloor$, for which we have
\begin{align}
\fI^m \fD^s g
=
\fI^m \fD^{\lfloor s \rfloor} \fI^{\lfloor s \rfloor - s}g
&=
\fI^m \fD^m \fD^{\lfloor s \rfloor - m}\fI^{\lfloor s \rfloor - s}g
=
\fD^{\lfloor s \rfloor - m}\fI^{\lfloor s \rfloor - m - (s - m)}g
\\
&= 
\fD^{\lfloor s - m \rfloor}\fI^{\lfloor s - m \rfloor - (s - m)}g
=
\fD^{s - m} g.
\end{align}
Here, we used the definition of $\fD^s$, the additivity of 
integer powers of $\fD$, and the aforementioned analog to 
Proposition~\ref{prop:ID-DI-s}\ref{prop:ID-DI-s:ID}. This completes the 
proof that $u$ is a mild solution in the sense of 
Definition~\ref{def:mildsol-Dirichlet}.
\end{proof}
The proof of Theorem~\ref{thm:extension} involves expressing 
an integral in terms of 
fractional binomial coefficients, given by~\cite[Equations~(1.2.6), 
(5.2.4) and~(5.2.5)]{Olver2010}:
\begin{equation}\label{eq:frac-binom-coeff}
\binom{\alpha}{k}
\coloneqq 
\frac{1}{k!}
\prod_{\ell = 0}^{k-1}(\alpha - \ell)
=
\frac{\Gamma(\alpha + 1)}{k! \, \Gamma(\alpha - k + 1)},
\quad 
\alpha \in (0,\infty), \, k \in \N_0.
\end{equation}
Since the author is not aware of a direct reference for the following integral 
identity, 
a 
proof is presented below for the sake of self-containedness.
\begin{lemma}\label{lem:hgf-integral}
For $\alpha \in (0,1)$, $a, b \in (0,\infty)$
and $n \in \N_0$ we have
\begin{equation}\label{eq:nasty-integral}
\begin{aligned}
	\frac{\sin(\pi \alpha)}{\pi}
	\int_0^{\nicefrac{a}{b}} &\frac{\tau^{-\alpha}(a - b 
		\tau)^{\alpha+n-1}}{\tau+1} \rd \tau
	\\&=
	(a + b)^{\alpha + n - 1}
	-
	\sum_{k=1}^{n} 
	\binom{\alpha + n - 1}{n - k}
	a^{n-k} b^{\, k+\alpha-1}.
\end{aligned}
\end{equation}
\end{lemma}
\begin{proof}
By the change of variables $\sigma \coloneqq \frac{b}{a}\tau$
and~\cite[Equation~(5.5.3)]{Olver2010}, 
the validity of the identity~\eqref{eq:nasty-integral} is
equivalent to that of
\begin{equation}\label{eq:nasty-integral-changeofvars}
\begin{aligned}
	\frac{1}{\Gamma(1-\alpha)\Gamma(\alpha)}
	\int_0^1 &\frac{\sigma^{-\alpha}(1-\sigma)^{\alpha+n-1}}{\sigma + 
		\frac{b}{a}}\rd \sigma
	\\&=
	a^{1-n} b^{-\alpha} (a + b)^{\alpha+n-1}
	-
	\sum_{k=1}^{n} 
	\binom{\alpha + n - 1}{n - k} \biggl(\frac{b}{a}\biggr)^{\! k-1}.
\end{aligned}
\end{equation}
We will verify this identity using induction on $n \in \N_0$. The base case
$n = 0$ is a consequence of~\cite[Equations~(5.12.4)
and~(5.12.1)]{Olver2010}:
\begin{equation}
\frac{1}{
	\Gamma(1-\alpha)\Gamma(\alpha)}\int_0^1 
\frac{\sigma^{-\alpha}(1-\sigma)^{\alpha-1}}{\sigma + \frac{b}{a}}\rd 
\sigma
=
\left( 1 + \tfrac{b}{a}\right)^{\! \alpha - 1} 
\left(\tfrac{b}{a}\right)^{\!-\alpha}
=
a b^{-\alpha} (a+b)^{\alpha-1}.
\end{equation}
Now suppose that~\eqref{eq:nasty-integral-changeofvars} holds for a given $n
\in \N_0$. In order to establish the identity for $n + 1$, we write 
$1 - \sigma = 1 + \frac{b}{a} - (\sigma + \frac{b}{a})$ and
apply the induction hypothesis
and~\cite[Equation~(5.12.1)]{Olver2010}, respectively, to the resulting two 
integrals:
\begin{equation}
\begin{aligned}
	&\frac{1}{\Gamma(1-\alpha)\Gamma(\alpha)}
	\int_0^1 \frac{\sigma^{-\alpha}(1-\sigma)^{\alpha+n}}{\sigma + 
		\frac{b}{a}}\rd \sigma
	\\&=
	\begin{aligned}[t]
		&\frac{1 + \frac{b}{a}}{\Gamma(1-\alpha)\Gamma(\alpha)}
		\int_0^1 \frac{\sigma^{-\alpha}(1-\sigma)^{\alpha+n-1}}{\sigma + 
			\frac{b}{a}}\rd \sigma
		\\&-
		\frac{1}{\Gamma(1-\alpha)\Gamma(\alpha)}
		\int_0^1 \sigma^{-\alpha}(1-\sigma)^{\alpha+n-1}\rd \sigma
	\end{aligned}
	\\
	&=
	a^{-n} b^{-\alpha} (a + b)^{\alpha + n}
	-
	\left(\! 1 + \frac{b}{a}\right)
	\sum_{k=1}^{n} 
	\binom{\alpha + n - 1}{n - j} \biggl(\frac{b}{a}\biggr)^{\! j-1}
	-
	\frac{\Gamma(\alpha + n)}{n! \, \Gamma(\alpha)}.
\end{aligned}
\end{equation}
For the latter two terms, we have
\begin{align}
&\left(\! 1 + \frac{b}{a}\right)
\sum_{k=1}^{n} 
\binom{\alpha + n - 1}{n - k} \biggl(\frac{b}{a}\biggr)^{\! k-1}
+
\frac{\Gamma(\alpha + n)}{n! \, \Gamma(\alpha)}
\\
&=
\sum_{k=1}^{n} 
\binom{\alpha + n - 1}{n - k} \biggl(\frac{b}{a}\biggr)^{\! k-1}
+
\sum_{k=0}^{n} 
\binom{\alpha + n - 1}{n - k} \biggl(\frac{b}{a}\biggr)^{\! k}
\\
&=
\sum_{k=1}^{n} 
\biggl[
\binom{\alpha + n - 1}{n - k}
+
\binom{\alpha + n - 1}{n - k + 1}
\biggr]\biggl(\frac{b}{a}\biggr)^{\! k-1}
+
\biggl(\frac{b}{a}\biggr)^{\!\! n}
=
\sum_{k=1}^{n+1} 
\binom{\alpha + n}{n + 1 - k} \biggl(\frac{b}{a}\biggr)^{\! 
	k-1}\!\!.
\end{align}
Indeed, to obtain the second line we note that
$\frac{\Gamma(\alpha + n)}{n! \, \Gamma(\alpha)} 
= 
\binom{\alpha + n - 1}{n}$ by~\eqref{eq:frac-binom-coeff},
which is the term $k = 0$ of the second summation on the second line;
shifting its index of summation and splitting off the last term yields
the first expression on the third line.
The final step uses~\cite[Equation~(1.2.7)]{Olver2010} and the fact that 
$(\frac{b}{a})^n$ 
corresponds to the term $k = n + 1$ in the desired formula.
Putting the previous two displays together proves the induction step and 
thereby the lemma.
\end{proof}
\begin{remark}\label{rem:hypergeometric}
An alternative way to derive~\eqref{eq:nasty-integral} is by noting that 
the 
integral can be expressed in terms of a hypergeometric 
function~\cite[Equation~(15.6.1)]{Olver2010} to which one can apply the 
transformation formula~\cite[Equation~(15.8.2)]{Olver2010}. This results in 
a difference of two hypergeometric functions, whose 
definitions can be written out to respectively yield an infinite and a 
finite sum: The former is the fractional binomial expansion of 
$(a+b)^{\alpha + n - 1}$ and the latter consists of its first $n$ terms, 
and together this gives~\eqref{eq:nasty-integral}. In particular, we note 
that~\eqref{eq:nasty-integral} is formally equal to 
the tail of a fractional binomial series. The 
proof of Lemma~\ref{lem:hgf-integral} is more direct and avoids the need to 
address the convergence of an infinite series.
\end{remark}
\begin{proof}[Proof of Theorem~\ref{thm:extension}]
Fixing $t \in (t_0, \infty)$, the semigroup law implies
\begin{equation}\label{eq:Str-identity}
S(t-r) = S((t-t_0)(\tau + 1))S(t_0 - (t-t_0)\tau-r)
\end{equation}
for all $\tau \in (0, \infty)$ and $r \in (-\infty, t_0 - (t-t_0)\tau)$.
This identity,
followed by~\eqref{eq:semigroup-est}, H\"older's inequality
and equation~\eqref{eq:confHGF-incgamma} yields
\begin{align*}
&\frac{1}{\Gamma(s)}
\int_{0}^{\infty} 
\int_{-\infty}^{t_0 - (t-t_0)\tau}  
\biggl\| 
\frac{\tau^{-\{s\}}(t_0 - (t-t_0)\tau-r)^{\gamma-1}}{\tau+1} S(t-r) 
f(r)
\biggr\|_{X} 
\rd r 
\rd \tau
\\
&\quad \le 
M_0 \int_{0}^{\infty}
\frac{\tau^{-\{s\}}}{\tau + 1}e^{-w(t - t_0)(\tau + 1)}
\int_{-\infty}^{t_0 - (t-t_0)\tau}  
\| 
k_s(t_0 - (t-t_0)\tau -r) 
f(r)
\|_{X} 
\rd r 
\rd \tau
\\
&\quad \le
\frac{M_0 \pi }{\sin(\pi\{s\})}
\overline\Gamma(\{s\}, w(t-t_0))
\| k_s \|_{L^{p'}(0,\infty; \mathscr L(X))}\norm{f}{L^p(\R;X)} < \infty.
\end{align*}
This justifies the use of Fubini's theorem in the following:
\begin{align*}
&\int_0^\infty \frac{\tau^{-\{s\}}}{\tau + 1} S((t-t_0)(\tau + 1)) 
\fI^s f(t_0 - (t-t_0)\tau) \rd \tau
\\&=
\frac{1}{\Gamma(s)}
\int_0^\infty 
\int_{-\infty}^{t_0 - (t-t_0)\tau} 
\frac{\tau^{-\{s\}}(t_0 - (t-t_0)\tau-r)^{s-1}}{\tau+1} 
S(t - r)f(r)
\rd r
\rd \tau
\\&=
\frac{1}{\Gamma(s)}
\int_{-\infty}^{t_0} 
\biggl[
\int_0^{\frac{t_0 - r}{t - t_0}} 
\frac{\tau^{-\{s\}}(t_0 - r - (t-t_0)\tau)^{s-1}}{\tau+1}
\rd \tau
\biggr]
S(t - r)
f(r) \rd r,
\end{align*}
where we used~\eqref{eq:Str-identity} once more.
Lemma~\ref{lem:hgf-integral} and
equation~\eqref{eq:frac-binom-coeff} produce
\begin{equation}
\begin{aligned}
	&\frac{\sin(\pi\{s\})}{\pi}
	\int_0^{\frac{t_0 - r}{t - t_0}} 
	\frac{
		\tau^{-\{s\}}
		(t_0 - r - (t-t_0)\tau)^{\{s\} + \lfloor s \rfloor -1}
	}{\tau+1}
	\rd \tau
	\\
	&\qquad =
	(t-r)^{s-1}
	-
	\sum_{k=1}^{\lfloor s \rfloor} \frac{\Gamma(s)}{(\lfloor s \rfloor 
		- k)! \, \Gamma(\{s\} + k)}
	(t_0 - r)^{\lfloor s \rfloor - k} (t-t_0)^{k + \{s\} - 1}.
\end{aligned}
\end{equation}
The previous two displays and the identity 
$S(t - r) = S(t - t_0) S(t_0 - r)$ yield
\begin{equation}
\begin{aligned}
	&\frac{\sin(\pi\{s\})}{\pi}
	\int_0^\infty 
	\frac{\tau^{-\{s\}}}{\tau + 1} S((t-t_0)(\tau + 1)) 
	\fI^s f(t_0 - (t-t_0)\tau) 
	\rd \tau
	\\
	&\quad = 
	\begin{aligned}[t]
		&-
		\sum_{k=1}^{\lfloor s \rfloor} 
		\frac{(t-t_0)^{\{s\} + k - 1} 
			S(t-t_0)}{\Gamma(\{s\} + k)}
		\frac{1}{(\lfloor s \rfloor - k)!}
		\int_{-\infty}^{t_0} (t_0 - r)^{\lfloor s \rfloor - k} S(t_0-r) 
		f(r) \rd r
		\\
		&+ 
		\frac{1}{\Gamma(s)} \int_{-\infty}^{t_0} (t-r)^{s-1} 
		S(t-r) f(r) \rd r
	\end{aligned}
	\\
	&\quad= 
	-
	\sum_{k=1}^{\lfloor s \rfloor} 
	\frac{(t-t_0)^{\{s\} + k - 1} 
		S(t-t_0)}{\Gamma(\{s\} + k)}
	\fI^{\lfloor s \rfloor - k + 1} f(t_0)
	+ 
	\fI^s \widetilde{f}(t)
\end{aligned}
\end{equation}
which is precisely~\eqref{eq:Is-ftilde}. The final assertion follows from 
Proposition~\ref{prop:ID-DI-s}\ref{prop:ID-DI-s:ID}.
\end{proof}

\section{Comparison to Riemann--Liouville and Caputo Cauchy problems}
\label{section:RL-C-ACP}
%
In this section, we compare the Dirichlet 
problem~\eqref{eq:natural-Dirichlet-Ds} 
to fractional-order abstract Cauchy problems of the form
\begin{equation}
(\partial_t + A)^s u(t) = 0, \quad t \in (t_0, \infty),
\end{equation}
augmented with initial conditions which depend on the interpretation of 
the abstract space--time operator $(\partial_t + A)^s$,
acting on functions $u \from J \to X$ with $J = (t_0, \infty)$ for 
$t_0 \in \R$ (instead of $J = \R$ as in the previous sections). More 
precisely, we will interpret 
$(\partial_t + A)^s$ as a Riemann--Liouville or Caputo type fractional 
parabolic derivative, respectively, on $L^p(J; X)$ 
for $p \in [1,\infty]$, and determine the corresponding initial conditions 
and mild solution formulae, see Definitions~\ref{def:RL-ACP} 
and~\ref{def:C-ACP}.

For fractional time derivatives $\partial_t^s$, i.e., the case $A = 0$, 
the resulting solution concepts are well-known and commonly studied, see for 
instance~\cite[Chapter~3]{KST2006}. In this case, the analog 
to~\eqref{eq:def-Igamma} has less 
favorable 
mapping properties~\cite[Section~2.3]{KST2006} and the well-posedness 
of~\eqref{eq:mildsol-Dirichlet} is less clear.
We will show that, as in the case $A = 0$,
the lowest-order term of the solution to the Riemann--Liouville type initial 
value problem has a singularity at $t_0$ in 
general, whereas 
the Caputo initial value problem yields the same solution for any two 
$s_1, s_2 \in (n, n+1)$ for $n \in \N_0$.
In contrast, the solution from 
Definition~\ref{def:mildsol-Dirichlet} is continuous at $t_0$ under mild 
assumptions on $g$ or $(S(t))_{t\ge0}$ and changes for all choices of $s \in 
(0,\infty)$.

Firstly, let us recall the Riemann--Liouville type fractional 
parabolic 
integral $\fI_{t_0}^s$ on $L^p(J; X)$ defined by~\eqref{eq:def-RL-int}.
Then, the Riemann--Liouville and Caputo type fractional 
parabolic \emph{derivatives} are respectively defined by
\[
\mathfrak D^s_{\mathrm{RL}} \coloneqq (\partial_t + A)^{\lceil s \rceil} 
\mathfrak 
I_{t_0}^{\lceil s \rceil - s}
\quad \text{and}\quad 
\mathfrak D^s_{\mathrm{C}} \coloneqq 
\mathfrak 
I_{t_0}^{\lceil s \rceil - s} (\partial_t + A)^{\lceil s \rceil}
\]
on their maximal domains.
In order to derive mild solution formulae for $L^p$-solutions to the equations $\fD^s_{\mathrm{RL}} u = 0$ and $\fD^s_{\mathrm{C}} u = 0$, we proceed analogously to~\cite[Chapter~3]{KST2006} and express $\fI_{t_0}^s \fD^s_{\mathrm{RL}} u$ and $\fI_{t_0}^s \fD^s_{\mathrm{C}} u$ in terms of initial data from $u$ (compare with Proposition~\ref{prop:InDn} and Theorem~\ref{thm:extension}), so that applying $\fI^s_{t_0}$ on both sides of the equations motivates the definitions.
The integer-order case $s = n 
\in \N$, 
where 
\[
\fD^n = \fD_{\mathrm{RL}}^n = \fD_{\mathrm{C}}^n = (\partial_t + A)^n,
\]
was treated in Section~\ref{sec:Lp-implies-mild:integer-order}.
From these results, we derive the following proposition regarding 
$\fD_{\mathrm{RL}}^s$ and
$\fD_{\mathrm{C}}^s$ for fractional $s \in (0,\infty)\setminus\N$:
\begin{proposition}\label{prop:Int-Deriv-gamma-RL}
Let Assumption~\ref{ass:bdd-semigroup}\ref{ass:bdd-semigroup:bdd}
be satisfied. If $s \in (0,\infty) \setminus \N$, $t_0 \in \R$ and $p \in 
[1,\infty]$ and $u \in 
\dom{\fD_{\mathrm{RL}}^s}$, then for almost all $t \in J \coloneqq 
(t_0,\infty)$:
\begin{equation}
\begin{aligned}
	\mathfrak I_{t_0}^{s}  \mathfrak D_{\mathrm{RL}}^s u(t)
	&= 
	u(t)
	-
	\frac{(t-t_0)^{\{s\} - 1}}{\Gamma(\{s\})} 
	S(t - t_0) 
	\fI_{\mathrm{RL}}^{1 - \{s\}} u(t_0)
	\\&\quad -
	\sum_{k=1}^{\lfloor s \rfloor} 
	\frac{(t - t_0)^{k + \{s\} - 1}}{\Gamma(k + \{s\})} 
	S(t - t_0) 
	\fD_{\mathrm{RL}}^{k + \{s\} - 1} u(t_0).
\end{aligned}
\end{equation}
If
$u \in \dom{\fD_{\mathrm{C}}^s}$ is such that
$u \in C^j(\clos J; \mathsf D(A^{n-1-j}))$
for all $j \in \{0,\dots, n-1\}$, then we have for almost all $t \in J 
\coloneqq (t_0,\infty)$:
\begin{equation}
\mathfrak I_{t_0}^{s}  \mathfrak D_{\mathrm{C}}^s u(t)
= 
u(t)
-
\sum_{k=0}^{\lfloor s \rfloor} \frac{(t - t_0)^k}{k!} 
\overline\Gamma(\lceil s \rceil-k, (t - t_0)A)u^{(k)}(t_0)
\quad 
\text{a.e.}
\end{equation}
\end{proposition}
\begin{proof}
For the sake of notational convenience we only present the case $t_0 = 0$.
The definition of $\fD_{\mathrm{RL}}^s$, along with 
Propositions~\ref{prop:frac-par-int-props}\ref{prop:frac-par-int-props:semigroup}--\ref{prop:frac-par-int-props:invariant-funcs}
and~\ref{prop:InDn}, yields
\begin{align}
\fI^{\lceil s \rceil - s}_0 \fI^s_0 \fD_{\mathrm{RL}}^s u
&=
\fI^{\lceil s \rceil}_0 \fD_{\mathrm{RL}}^s u
=
\fI^{\lceil s \rceil}_0 (\partial_t + A)^{\lceil s \rceil} 
\fI^{\lceil	s \rceil - s}_0 u
\\
&=
\fI^{\lceil s 
	\rceil - s}_0 u
-
\sum_{k=0}^{n-1} 
\frac{(\,\cdot\,)^k}{k!} S(\,\cdot\,) [(\partial_t + A)^k \fI^{\lceil s 
	\rceil - s}_0 u](0)
\\
&=
\fI^{\lceil s 
	\rceil - s}_0 
\biggl[
u
-
\sum_{k=0}^{n-1} 
\frac{(\,\cdot\,)^{k-\lceil s \rceil + s}
	S(\,\cdot\,)}{\Gamma(k + s -\lceil s 
	\rceil + 1)} 
[(\partial_t + A)^k \fI^{\lceil s 
	\rceil - s}_0 u](0)\biggr]
\end{align}
for any $u \in \dom{\fD_{\mathrm{RL}}^s}$.
The first assertion then follows from Proposition~\ref{prop:ID-DI-s} and 
the injectivity of $\fI^{\lceil s \rceil -s }_0$.

If $u \in \dom{\fD_{\mathrm{C}}^s}$, combining the definition with 
Proposition~\ref{prop:frac-par-int-props}\ref{prop:frac-par-int-props:semigroup}
produces
\begin{equation}
\mathfrak I^{s}_0 \mathfrak D_{\mathrm{C}}^s u
= 
\mathfrak I^{s}_0 \mathfrak I^{\lceil s \rceil - s}_0 
(\partial_t + A)^{\lceil s 
	\rceil} u
=
\mathfrak I^{\lceil s \rceil}_0 
(\partial_t + A)^{\lceil s 
	\rceil} u,
\end{equation}
so that the result follows from Proposition~\ref{prop:InDn}
and the discussion below it, in particular 
equation~\eqref{eq:Int0-dtAn-Gamma}.
\end{proof}

Note that
$\fI_{0}^{1 - \{s\}} u$ need not vanish at $t_0 = 0$. Indeed, even if it 
is continuous, it may not 
satisfy~\eqref{eq:def-RL-int} pointwise, as evidenced by the example 
$u \coloneqq k_{\{s\}} \otimes x$ 
for $p \in [1, \frac{1}{\{s\}-1})$ and $x \in \clos{\dom A} \setminus \{0\}$, 
see 
Proposition~\ref{prop:frac-par-int-props}\ref{prop:frac-par-int-props:invariant-funcs}.

Proposition~\ref{prop:Int-Deriv-gamma-RL} motivates the following definition of 
the
Riemann--Liouville fractional abstract Cauchy type problem and its 
corresponding solution.

\begin{definition}\label{def:RL-ACP}
Let Assumption~\ref{ass:bdd-semigroup}\ref{ass:bdd-semigroup:bdd}
be satisfied. For any $s \in (0,\infty) \setminus \N$ and $t_0 \in \R$,
the mild solution to the Riemann--Liouville abstract Cauchy type 
problem 
\begin{equation}\label{eq:RL-ACP}
\left\lbrace 
\begin{aligned}
	\fD^{s}_{\mathrm{RL}} u(t) &= 0,  &&t \in J \coloneqq (t_0, 
	\infty), \\
	\fI^{1 - \{s\}}_{t_0} u(t_0) 
	&= x_0 \in X, && \\
	\fD^{k + \{s\} - 1}_{\mathrm{RL}} u(t_0) &= x_k \in 
	\clos{\dom{A}}, &&
	k \in 
	\{1,\dots,\lfloor s \rfloor\},
\end{aligned}
\right.
\end{equation}
is the function
$u_{\mathrm{RL}} \in C(J; X)$ defined by
\begin{equation}\label{eq:uRL}
u_{\mathrm{RL}}(t)
\coloneqq
\sum_{k=0}^{\lfloor s \rfloor} 
\frac{(t-t_0)^{k + \{s\} - 1}}{\Gamma(k + \{s\})} 
S(t - t_0) x_k, 
\quad 
t \in J.
\end{equation}
\end{definition}
Compared with
Definition~\ref{def:mildsol-Dirichlet}, we first note that the terms $k \in 
\{1, \dots, \lfloor s \rfloor \}$ in~\eqref{eq:uRL} are almost identical to 
those of~\eqref{eq:Es}, up to the difference between taking Riemann--Liouville 
type fractional parabolic derivatives of the function $u$ defined on $J$ and 
Weyl type derivatives of $g$ defined on $\R\setminus J$. The remaining term, on 
the other hand, differs significantly. In~\eqref{eq:RL-ACP}, we see that 
$x_0$ is the prescribed value of $\fI^{1 - \{s\}}_{t_0} u$ at $t_0$ and
$u_{\mathrm{RL}}$ is continuous at $t_0$ if and only if $x_0 = 0$, in view of 
the singularity occurring there for $x_0 \neq 0$. In contrast, the 
solution to~\eqref{eq:natural-Dirichlet-Ds} given by 
Definition~\ref{def:mildsol-Dirichlet} is bounded by 
Proposition~\ref{prop:mildsol-Dirichlet-bound-cont}, does in fact prescribe the 
value $u(t_0) = g(t_0)$ and is continuous on $\bbR$ under some further 
regularity assumptions.

The following definition of a Caputo type initial value 
problem and corresponding solution can also be derived from 
Proposition~\ref{prop:Int-Deriv-gamma-RL}:
\begin{definition}\label{def:C-ACP}
Let Assumption~\ref{ass:bdd-semigroup}\ref{ass:bdd-semigroup:bdd}
be satisfied. For any $s \in (0,\infty) \setminus \N$ and $t_0 \in \R$,
the mild solution to the Caputo abstract Cauchy problem
\begin{equation}\label{eq:C-fracACP}
\left\lbrace 
\begin{aligned}
	\fD^s_{\mathrm{C}} u(t) &= 0, &t &\in J \coloneqq (t_0, \infty), 
	\\
	u^{(k)}(t_0) &= x_k \in \dom{A^{\lfloor s \rfloor - k}}, 
	&
	k &\in \{0,\dots,\lfloor s \rfloor\}.
\end{aligned}
\right.
\tag{C-ACP$_s$}
\end{equation}
is the function
$u_{\mathrm{C}} \in C(J; X)$ defined by
\begin{equation}
u_{\mathrm{C}}(t)
\coloneqq
\sum_{k=0}^{\lfloor s \rfloor} \frac{(t - t_0)^k}{k!} 
\overline\Gamma(\lceil s \rceil-k, (t - t_0)A)x_k, 
\quad 
t \in J.
\end{equation}
\end{definition}
Note that this definition has the same form as 
the integer-order abstract Cauchy problem from 
Definition~\ref{def:mildsol-Dirichlet}, i.e., formula~\eqref{eq:Es-n}. 
Analogously, for sufficiently regular 
$x_k$ 
or $(S(t))_{t\ge0}$, this solution allows for the specification of the 
value of $u_{\mathrm{C}}(t_0)$.
However, in contrast to the  solution in 
the sense of Definition~\ref{def:mildsol-Dirichlet}, we observe that the form 
of $u_{\mathrm{C}}$ only changes ``discretely in $s$,'' i.e., the solutions for 
any two $s_1, s_2 \in (n, n+1)$, $n \in \N_0$ are given by the same formula.

\section*{Acknowledgments}

The author acknowledges helpful discussions with Kristin Kirchner and Wolter Groenevelt which contributed to the formulations of Theorem~\ref{thm:extension} and Remark~\ref{rem:hypergeometric}, respectively. 
Moreover, the author thanks Mark Veraar, Jan van Neerven and an anonymous reviewer for carefully reading the manuscript and providing valuable comments.

\bibliographystyle{siam}
\bibliography{fractional-parabolic-dirichlet.bib}

\end{document}